\documentclass[a4paper,11pt, oneside]{article}

\usepackage[english]{babel}
\usepackage[T1]{fontenc}
\usepackage{enumerate}
\usepackage{amsmath}
\usepackage{amsthm}
\usepackage{amssymb}
\usepackage{subfig}
\usepackage{tikz}

\newcommand{\Qp}{\mathbb{Q}_p}

\newcommand{\Zp}{\mathbb{Z}_p}

\newcommand{\Fp}{\mathbb{F}_p}
\newcommand{\Cp}{\mathbb{C}_p}

\newcommand{\Q}{\mathbb{Q}}
\newcommand{\R}{\mathbb{R}}

\newcommand{\C}{\mathbb{C}}
\newcommand{\Z}{\mathbb{Z}}
\newcommand{\N}{\mathbb{N}}

\newcommand{\rightleftarrow}{\leftrightarrow}

\newtheorem{theorem}{Theorem}[section]

\newtheorem{prop}[theorem]{Proposition}

\newtheorem{lemma}[theorem]{Lemma}
\newtheorem{cor}[theorem]{Corollary}
\newtheorem{definition}[theorem]{Definition}
\newtheorem*{Fact}{Fact}
\newtheorem{claim}{Claim}

\theoremstyle{remark}
\newenvironment{Remark}{\begin{trivlist}\item[\hskip \labelsep {\bfseries Remark.}]}{\end{trivlist}}
\newenvironment{acknowledgement}{\begin{trivlist}\item[\hskip \labelsep {\bfseries Acknowledgements.}]}{\end{trivlist}}
\newtheorem{Exp}{Example}[section]

\newenvironment{notation}{\begin{trivlist}\item[\hskip \labelsep {\bfseries Notations.}]}{\end{trivlist}}

\newcommand{\Wsyst}[1]{ [\![{#1}]\!]}

\title{Effective model-completeness for $p$-adic analytic structures}
\author{Nathana\"el Mariaule}
\date{}
\begin{document}
\maketitle

\begin{abstract}In \cite{Denef-vdD}, J. Denef and L. van den Dries prove that the theory of the ring of $p$-adic integers admits the elimination of quantifiers in the language of $p$-adic restricted analytic functions expanded by a division symbol. In this paper, we are interested in restrictions of this language: Let $F$ be any family of $p$-adic restricted analytic functions, we construct an expansion of $F$ so that the theory of the ring of $p$-adic integers is model-complete in the corresponding language. Next, we give conditions on $F$ so that the model-completeness is effective. Finally, we apply our results in the context of the $p$-adic exponential ring. 
\end{abstract}

\section{Introduction}

\par Let $\mathcal{L}_{an}$ be the expansion of the language of rings by unary predicates for the set of $n$th powers and function symbols for all restricted $p$-adic analytic function (i.e.\ functions defined by power series convergent on $\Zp$). The model theory of $\Zp$ in this language was first considered by J. Denef and L. van den Dries in \cite{Denef-vdD}. In particular, they proved that this theory admits the elimination of quantifiers if we expand $\mathcal{L}_{an}$ by a symbol of division $D$. $\mathcal{L}_{an}$-definable sets are the $p$-adic subanalytic sets. In this paper, we consider reduct of this language i.e.\ let $F$ be any family of  restricted analytic functions, we consider $\mathcal{L}_F$ the expansion of the language of rings by function symbols for each $f\in F$ and predicates for the set of $n$th powers. 

\par A careful inspection of \cite{Denef-vdD} (one needs some new arguments) gives the following: if $F$ is a Weierstrass system (meaning roughly that it is closed under Weierstrass division, see section \ref{Weierstrass system}), then the theory of $\Zp$ admits quantifier elimination in the language $\mathcal{L}_F$ expanded by the symbol of division. This was developed in \cite{Cluckers-Lipshitz} in much greater generality. We recall this in section \ref{Weierstrass system} and adapt the result to our setting.
\par For general $F$, it seems very unlikely that the theory of $\Zp$ in this language also admits quantifier elimination. In this paper, we give conditions so that this theory admits the next best thing after quantifier elimination: strong model-completeness.

\par A. Macintyre \cite{Macintyre4} proved a model-completeness result for $F=\{(1+p)^x\}$.
In this case, $\mathcal{L}_F$ induces a structure of exponential ring on $\Zp$ (which can be thought as the $p$-adic equivalent of the structure $(\R,+, -, \cdot, 0,1,<, exp\mathord{\upharpoonright}_{[-1,1]})$). A. Wilkie proved that the theory the field of reals with restricted exponentiation is model-complete in \cite{Wilkie}. In the $p$-adic case, the ideas of the proof of Macintyre goes back to  \cite{vdDries4}.
\par In this paper \cite{vdDries4}, L. van den Dries proves the strong model-completeness of the structure $(\R, +, -, \cdot,$ $ exp\mathord{\upharpoonright}_{[-1,1]}, \sin\mathord{\upharpoonright}_{[-1,1]}, \cos\mathord{\upharpoonright}_{[-1,1]})$. A key property is that the structure $(\C, +, -, \cdot,$ $ exp,$ $ \sin, \cos)$ is definable in $(\R, +, -, \cdot,$ $ exp, \sin, \cos)$ (where all functions are restricted to a compact set). In the $p$-adic case, the algebraic closure is an extension of infinite degree and therefore is not definable. But it is sufficient to interpret the natural structure attached to the valuation ring of any finite algebraic extension i.e.\ we want the structure $(V,+, -, \cdot, 0,1, f; f\in F)$ to be $\mathcal{L}_F$-definable in $\Zp$ for any $V$ valuation ring of a finite algebraic extension $K$. In general it may not be the case and so we expand $F$ by a family $\widetilde{F}$ of functions called the \emph{decomposition functions} so that $(V,+, -, \cdot, 0,1, f; f\in \widetilde{F})$  becomes definable in $(\Zp,+, -, \cdot, 0,1, f; f\in \widetilde{F})$ (see section \ref{Decomposition functions}). For this expansion of the language, finite algebraic extensions are definable and the proof of model-completeness follows roughly by mimicking the proof in the real case. The first main theorem is:
{
\renewcommand{\thetheorem}{\ref{strong model-completeness}}
\begin{theorem}
Let $F$ be a family of restricted analytic functions. Assume that the set of $\mathcal{L}_F$-terms is closed under derivation. Let $\widetilde{F}$ be the extension of $F$ by the decomposition functions of $f$ for each $f\in F$. Then, $\mathbb{Z}_{p,\widetilde{F}}$ is strongly model-complete in $\mathcal{L}_{\widetilde{F}}$.
\end{theorem}
\addtocounter{theorem}{-1}
} 

\par The main idea of the proof is to construct a Weierstrass system $W_F$ generated by $F$.  We give the definition of this Weierstrass system in section \ref{Weierstrass system generated by}. If $F$ is closed under decomposition function, we show that any element $f\in W_F$ is strongly definable in $\mathcal{L}_F$ under the condition that the set of $\mathcal{L}_F$-terms is closed under derivation. This will be done in section \ref{Strong model-completeness}. We combine these existential definitions with the quantifier elimination in the language with all functions of the Weierstrass system $W_F$ to get the result of strong model-completeness in Theorem \ref{strong model-completeness}.

\par In the second part of this paper, we are interested in decidability issues. We consider the following problem: to give conditions on $F$ so that Theorem \ref{strong model-completeness} is effective i.e.\ there is an algorithm which takes for entry a formula in our language and returns an existential formula equivalent to it.
\par Most of the proof of Theorem \ref{strong model-completeness} is already effective but there is some issues related to the following problem: let $f=\sum a_I(\overline{X})\overline{Y}^I$ in the Weierstrass system $W_F$. Let $\mathcal{I}$ be the ideal generated by the elements $a_I$ in $\Zp\{\overline{X}\}$, the ring of restricted power series. We need to determine an integer $d(f)$ so that $\mathcal{I}$ is generated by $d(f)$ elements (we will also need some additional properties that will appear in section \ref{Weierstrass system}). 

\par In \cite{Denef-vdD}, they obtain this bound $d(f)$ as $\Zp\{\overline{X}\}$ is Noetherian. The existence is also guaranteed in our case but we need an explicit computation. Let $g=\sum b_i(\overline{X})Y^i\in W_F$. It is known that for all $\overline{x}\in\Zp^k$, there is a bound $S(g)$ on the number of zeros (counting multiplicities) of $g(\overline{x},Y)$ in the valuation ring of $\Cp$ independent of the choice of $\overline{x}$ whenever this number is finite (it is a consequence of Weierstrass preparation theorem in the style of \cite{Denef-vdD}). In section \ref{Weierstrass system}, we show that $d(f)$ can be determined in an effective way in terms of $S(h_1),\cdots, S(h_k)$ for some $h_i\in W_F$ constructed from $f$ in a explicit way.

\par In section \ref{Effective Weierstrass system}, we discuss this issue and show that it is sufficient to compute $S(g)$ for all $g\in W_F$. Furthermore, we will see that $S(g)$ is determined by the number of zeros of a system of equations of $\mathcal{L}_F$-terms. Given such a system, we will prove a counting point theorem in section \ref{Effective bound} using results of tropical analytic geometry due to J. Rabinoff \cite{Rabinoff}. This counting point theorem relates the number of zeros of a system to the integers $d(h)$ where $h$ are now terms. So, assuming that we can compute $d(h)$ for all $h$ terms in our language, we can compute $d(f)$ for all $f\in W_F$. One needs further assumptions on $F$ to apply the counting point theorem: first, we require that any $f\in F$ is overconvergent i.e.\ it is convergent on a ball $B$ that strictly contains the valuation ring (in $\Cp$). Second, we ask that $d(\widetilde{f})$ is computable where $\widetilde{f}(\overline{x})=f(t\overline{x})$ for $t\in B$ of negative valuation. In that case, we say that $W_F^{(0)}$ has an extended effective Weierstrass bound. Under these assumptions, we prove:
{
\renewcommand{\thetheorem}{\ref{effective model completeness of L_F}}
\begin{theorem}Let F be an effective family of restricted analytic functions such that the set of $\mathcal{L}_F$-terms is closed under derivation. Let $\widetilde{F}$ be the extension of $F$ by all decomposition functions of elements in $F$. Assume that each $\mathcal{L}_{\widetilde{F}}$-term is overconvergent and that $W_F^0$ has an extended effective Weierstrass bound.
\par Then, the theory of $\mathbb{Z}_{p,\widetilde{F}}$ is effectively strongly model-complete in the language $\mathcal{L}_{\widetilde{F}}$. 
\end{theorem}
\addtocounter{theorem}{-1}
} 

\par Finally, in section \ref{exponential ring} we apply our result in the case $F=\{(1+p)^x\}$. In that case, the decomposition functions are polynomial combinations of exponential terms of type $e^{\alpha^i x}$ where $\alpha$ is in a suitable algebraic extension of $\Qp$. By our first main theorem, the theory of the ring of $p$-adic integers in the language of exponential rings extended by these decomposition functions is model-complete. That particular case was proved by A. Macintyre in \cite{Macintyre4}. We extend this model-completeness : We prove that this theory is effectively model-complete. In a next paper, the author will prove that this theory is decidable assuming a $p$-adic version of Schanuel's conjecture.

\begin{notation}Within this text, $\Qp$ will denote the field of $p$-adic numbers. We will denote the $p$-adic valuation by $v$. $\Cp$ will denote the $p$-adic completion of the algebraic closure of $\Qp$ and $\mathcal{O}_p$ its valuation ring. Given a ring $A$, we denote the set of nonzero elements by $A^*$ and the set of units by $A^\times$. If $K$ is a field, we denote its algebraic closure by $K^{alg}$.  The set of restricted power series is denoted by 
$$\Zp\{\overline{X}\}=\left\{\sum_I a_I \overline{X}^I\left|\vphantom{\Big\{}\right. a_I\in  \Zp,\ v(a_I)\rightarrow \infty \right\} $$
where $\overline{X}=(X_1,\cdots, X_n)$ and we use multi-index notation.
\end{notation}

\section{Weierstrass system and quantifier elimination}\label{Weierstrass system}

\begin{definition}
Let $f(\overline{X},Y)\in\Zp\{\overline{X},Y\}$. Let $\pi$ be the canonical projection $\Zp\{\overline{X},Y\}\rightarrow \Fp[\overline{X},Y]$. We say that $f$ is \emph{regular in $Y$ of order $d$} if $\pi(f(\overline{X},Y))$ is a monic polynomial in $Y$ of degree $d$. We say that $f(\overline{X},\overline{Y})\in\Zp\{\overline{X},\overline{Y}\}$ is \emph{preregular} of order $K$ in $\overline{Y}$ if
$$\pi(f(\overline{X},\overline{Y}))=\sum_{I\leq K} a_I(\overline{X})\overline{Y}^I$$
and  $a_K(\overline{X})=1$ ($\leq$ is the lexicographic order on $\N^n)$.
\end{definition}

\begin{definition}\label{def W-sys}
A \emph{Weierstrass system} over $\Zp$ is a family of rings $\Zp\Wsyst{X_1,\cdots,X_n}$, $n\in\N$, such that for all $n$, the following conditions hold: 
\begin{enumerate}
\item $\Z[\overline{X}]\subseteq \Zp\Wsyst{\overline{X}}\subseteq \Zp\{\overline{X}\}$;
\item For all permutations $\sigma$ of $\{1,\cdots, n\}$, if $f(\overline{X})\in \Zp\Wsyst{\overline{X}}$, then $f(X_{\sigma(1)},\cdots, X_{\sigma(n)})\in \Zp\Wsyst{\overline{X}}$;
\item If $f\in \Zp\Wsyst{\overline{X}}$ has an inverse $g$ in $\Zp\{\overline{X}\}$, then $g\in \Zp\Wsyst{\overline{X}}$;
\item Let $g \in \Zp\Wsyst{\overline{X}}$. If $f\in \Zp\Wsyst{\overline{X}}$ is divisible by $g(0)$ in $\Zp\{\overline{X}\}$, then $f/g(0)\in \Zp\Wsyst{\overline{X}}$;
\item (Weierstrass division) If $f\in \Zp\Wsyst{X_1,\cdots , X_{n+1}}$ and $f$ is regular of order $d$ in $X_{n+1}$, then, for all $g\in \Zp\Wsyst{X_1,\cdots , X_{n+1}}$, there are $A_0,\cdots, A_{d-1}\in \Zp\Wsyst{\overline{X'}}$ (where $\overline{X'}=(X_1,\cdots , X_n)$) and $Q\in \Zp\Wsyst{\overline{X}}$  such that
$$ g(\overline{X}) = Q(\overline{X})\cdot f(\overline{X}) + \Big(X_{n+1}^{d-1}A_{d-1}(\overline{X'})+\cdots + A_0(\overline{X'})\Big).$$  
\end{enumerate}
\end{definition}
\begin{Remark}
\begin{enumerate}
	\item More general definitions of Weierstrass system can be found in the literature e.g.\ \cite{Cluckers-Lipshitz}. One want to have quantifier elimination in these languages (expanded by division symbols).  Note that in \cite{Cluckers-Lipshitz}, the Weierstrass system are required to have the so-called \emph{Strong Noetherian property}. This property is crucial for quantifier elimination. In our case, this property is always true as we will see later in this section (strong Noetherian property is implied by Proposition \ref{c-prop 1}). The author is not aware of any general setting where we would have quantifier simplification in the style of Theorem \ref{strong model-completeness}.
	\item Let $W$ be a Weierstrass system. Then, by closure under Weierstrass division, it follows that $W$ is closed under derivation and composition. 
Let $f(X),g(Y)\in W$ (we only prove the one variable case, the generalization should be obvious). Then, $H^2$ (resp. X-g(Y)) are regular in H of order 2 (resp. in X of order 1). So, by Weierstrass division,
$$f(X_1+H)= Q(X,H)H^2+[A_1(X)H+A_0(X)] $$
$$f(X) = Q'(X,Y)(X-g(Y)) + R(Y). $$
Identifying coefficients, one see that $A_1(X)=f'(X)$. Replacing $X$ by $g(Y)$ in the second equality shows that $f(g(Y))=R(Y)$. Therefore, $f'(X), f(g(Y))\in W$.
\end{enumerate}
\end{Remark}

 Let $\mathcal{L}_{Mac}=(+, -, \cdot,0,1,P_n; n\in \N)$ be the Macintyre's language for $p$-adically closed fields i.e.\ $+,-, \cdot,0,1$ are interpreted in $\Zp$ by the natural operations and $P_n$ is a unary predicate for the set of $n$th powers i.e.\
$$\Zp \vDash P_n(x)\mbox{ iff } \exists y\in \Zp\ x=y^n.$$
\par Fix a Weierstrass system $W=(\Zp\Wsyst{X_1,\cdots,X_n})_{n\in \N}$. Let $\mathcal{L}_{W}$ be the extension of the language $\mathcal{L}_{Mac}$  by function symbols $f$ for each $f\in \Zp\Wsyst{ X_1,\cdots,X_n}$ and $\mathcal{L}_{W}^D$ be the expansion of $\mathcal{L}_{W}$ by a division symbol $D$ interpreted in $\Zp$ by:
$$D(x,y)=\left\{\begin{array}{ll}
x/y&\mbox{if }v(x)\geq v(y)\mbox{ and } y\not=0\\
0 &\mbox{otherwise.}
\end{array}\right.$$
Let $\mathbb{Z}_{p,W}$ (resp. $\mathbb{Z}_{p,W}^D$) be the structure with underlying set $\Zp$ and natural interpretations for the symbol of $\mathcal{L}_W$ (resp. $\mathcal{L}_W^D$).
Then,
\begin{theorem}\label{W-EQ}
The theory of $\mathbb{Z}_{p,W}^D$  admits the elimination of quantifiers. 
\end{theorem}

\begin{definition}
Let $\mathcal{M}$ be a $\mathcal{L}$-structure with underlying set $M$. We say that $\mathcal{M}$ is \emph{strongly model-complete} if for any $\mathcal{L}$-formula $\Psi (\overline{y})$, there is an existential $\mathcal{L}$-formula $\exists \overline{x} \Phi (\overline{x},\overline{y})$, where $\Phi$ is quantifier-free, such that for all $\overline{a}\in M^n$, 
$$\mathcal{M} \vDash  \Psi (\overline{a}) \rightleftarrow \exists \overline{x} \Phi (\overline{x},\overline{a}),$$
and furthermore, for each $\overline{a}$ such that $M\vDash \Psi(\overline{a})$, there is a unique tuple $\overline{b}$ in $M^m$ such that $M\vDash \Phi(\overline{b},\overline{a})$.
\par A set $X$ is \emph{strongly definable} if 
$$X= \{ \overline{a}\in M^n \mid\ \mathcal{M}\vDash \exists \overline{y} \Phi ( \overline{a}, \overline{b}, \overline{y})\},$$ 
and, for each $\overline{a}\in X$, there is a unique tuple $\overline{c}$ in $M^m$ such that $M\vDash \Phi ( \overline{a}, \overline{b}, \overline{c})$.
 A function is strongly definable if its graph and the complement of its domain are strongly definable. A structure is strongly definable if its domain as well as functions, relations and constant symbols of the language are strongly definable.
\end{definition}  
\par Note that  the graph of the function $D$ is strongly definable in $\mathcal{L}_W$. So, as an immediate corollary of the above theorem, we have
\begin{cor}The theory of $\mathbb{Z}_{p,W}$  is strongly model-complete. 
\end{cor} 

We give now a proof of the theorem.
\begin{proof}
We follow the proof of quantifier elimination in \cite{Denef-vdD}.
By (1.3) from \cite{Denef-vdD}, it is sufficient to prove the following : for all $\Phi (\overline{X},Y_1,\cdots, Y_n)$ quantifier-free $\mathcal{L}_{W}$-formula, there exists $\Psi(\overline{X},Z_1,\cdots , Z_{n-1})$ quantifier-free $\mathcal{L}_W^D$-formula such that 
\begin{itemize}
	\item $\Zp\vDash (\exists \overline{Y}\ \Phi (\overline{X},\overline{Y}))\rightleftarrow (\exists \overline{Z}\ \Psi(\overline{X},\overline{Z}))$;
	\item In $\Psi$ the symbol $D$ is only applied to terms not involving the variables $\overline{Z}$.
\end{itemize} 
Let $\Phi (\overline{X},Y_1,\cdots, Y_n)$ be a quantifier-free $\mathcal{L}_{W}$-formula. Let $f(\overline{X},\overline{Y})=\sum_I a_I(\overline{X})\overline{Y}^I$ in $W$ occurring in $\Phi$. For convenience of the reader we recall the following lemma from \cite{Denef-vdD} :
\begin{Fact}[Lemma 1.4 in \cite{Denef-vdD}]\label{Denef-vdD lemma 1.4 bis}
Let $f(\overline{X},\overline{Y})=\sum a_I(\overline{X})\overline{Y}^I\in \Zp\{\overline{X},\overline{Y}\}$. Then, there is $d\in \N$ such that, for all $I$ with $|I|\geq d$ (where $|I|=i_1+\cdots+i_m$),
$$a_I(\overline{X}) = \sum_{|J|<d}b_{IJ}(\overline{X}) a_J(\overline{X}), $$
where $b_{IJ}(\overline{X})\in \Zp\{\overline{X} \}$ with $\|b_{IJ}(\overline{X})\|<1$ (where $\|\sum c_I\overline{X}^I\|=\min \{v(c_I)\}$) and $\|b_{IJ}\|\rightarrow 0$ as $|I|\rightarrow \infty$.
\end{Fact}
 Let $Z_f(\overline{X})$ be the formula $$\bigwedge_{|I|<d} a_I(\overline{X})=0,$$
with $d$ given by the above Fact.
Fix $\overline{x}\in \Zp^m$. Note that $\Zp \vDash Z_f(\overline{x}) \rightleftarrow (\forall \overline{Y}\ f(\overline{x},\overline{Y})=0)$. If $Z_f(\overline{x})$ does not hold then there is $J$ ($|J|<d$) such that $\mu_{J,f}(\overline{x})$ holds where $\mu_{J,f}$ is the formula:
$$a_J(\overline{X})\not=0\wedge \bigwedge_{I<J, |I|<d} |a_I(\overline{X})|\leq |a_J(\overline{X})| \wedge \bigwedge_{J<I, |I|<d}  |a_I(\overline{X})|< |a_J(\overline{X})|.$$
Assume that $\Zp\vDash \mu_{I,f}(\overline{x})$. Then by Fact \ref{Denef-vdD lemma 1.4 bis}, there are $b_{IJ}\in p\Zp\{\overline{X},\overline{Y}\}$ such that 
\begin{align*}
f=& \sum_{I<J, |I|<d} a_I\overline{Y}^I + a_J\overline{Y}^J + \sum_{J<I, |I|<d} a_I\overline{Y}^I\\
	&	 \sum_{|I|\geq d}\left\{\sum_{K<J, |I|<d} b_{IK}a_K + b_{IJ}a_J + \sum_{J<K, |K|<d} b_{IK}a_K\right\}\overline{Y}^I.
\end{align*}
Then one can divide by $a_J$ and replace the quotients $a_I/a_I$ by new variables $V_I$ or $pV_I$ according to whether $I<J$ or $J<I$. One can define a series $\widetilde{f}$
\begin{align*}
\widetilde{f}=& \sum_{I<J, |I|<d} V_I\overline{Y}^I + \overline{Y}^J + \sum_{J<I, |I|<d} pV_I\overline{Y}^I\\
	&	 \sum_{|I|\geq d}\left\{\sum_{K<J, |I|<d} b_{IK}V_K + b_{IJ} + \sum_{J<K, |K|<d} b_{IK}pV_K\right\}\overline{Y}^I.
\end{align*}
such that $\Zp\vDash \mu_{J,f}(\overline{x})\rightarrow f(\overline{x},\overline{Y})=a_J(\overline{x})\widetilde{f}(\overline{x}, \overline{v}(\overline{x}),\overline{Y})$, where $v_I(\overline{x})=D(a_I(\overline{x}),a_J(\overline{x}))$ if $I<J$ and $v_I(\overline{x})=D(a_I(\overline{x}),pa_J(\overline{x}))$ otherwise. Note that so far we have that $\widetilde{f}, b_{IJ}\in \Zp\{\overline{X},\overline{V},\overline{Y}\}$ but we will prove in Lemma \ref{Weierstrass Denef-vdD lemma 1.4} that we can take $\widetilde{f}, b_{IJ}\in W$.
\par By construction, $\widetilde{f}$ is preregular of order $J$ in $\overline{Y}$. Then up to a change of variables $T$ determined by $T(Y_i)=Z_i+Z_n^{d^n-i}$ if $i<n$ and $T(Y_n)=Z_n$, we have that $T(\widetilde{f})$ is regular of order $E:=j_n+j_{n-1}d+\cdots +j_1d^{n-1}$. Therefore by Weierstrass preparation theorem there exist $U, A_0,\cdots, A_{E-1}\in W$ such that 
$$T(\widetilde{f})= U(Z_n^E+A_{E-1}Z_n^{E-1}+\cdots + A_0).$$
This latter function is polynomial in $Z_n$. So, using similar transformations for any function in $\Phi$ (one can take $d$ independent of the choice of $f$ in $\Phi$), it allows to apply quantifier elimination for $p$-adically closed fields and to eliminate the quantifier attached to $Z_n$. We refer to \cite{Denef-vdD} (1.5) for the details.
\par Let us remark that except for the use of Fact \ref{Denef-vdD lemma 1.4 bis} and for the definition of $\widetilde{f}$, we only use the properties 1-5 of the definition of Weierstrass system. We will prove in the next lemma that we may assume that $\widetilde{f}$ and the $b_{IJ}$'s belong to the Weierstrass system $W$. This will complete the proof of the theorem.
\end{proof}

\begin{prop}\label{Weierstrass Denef-vdD lemma 1.4}
Let $f(\overline{X},\overline{Y})=\sum a_I(\overline{X})\overline{Y}^I\in \Zp\Wsyst{\overline{X},\overline{Y}}$. Then, there is $d\in \N$ such that, for all $I$ with $|I|\geq d$ (where $|I|=i_1+\cdots+i_m$),
$$a_I(\overline{X}) = \sum_{|J|<d}b_{IJ}(\overline{X}) a_J(\overline{X}), $$
where $b_{IJ}(\overline{X})\in \Zp\Wsyst{\overline{X} }$ with $\|b_{IJ}(\overline{X})\|<1$ (where $\|\sum c_I\overline{Y}^I\|=\min \{v(c_I)\}$) and $\|b_{IJ}\|\rightarrow 0$ as $|I|\rightarrow \infty$.
Furthermore, let $\widetilde{f}$ as defined above.

Then, $\widetilde{f}\in \Zp\Wsyst{\overline{X},\overline{Y},\overline{V}}$.

\end{prop}

The proof is based on \cite{Cluckers-Lipshitz}. In fact, it follows from the next Lemma \ref{c-lemma 1} together with Lemma 4.2.14 and Theorem 4.2.15 from \cite{Cluckers-Lipshitz}. On the other hand, in the second part of the paper we will be interested in an effective version of Theorem \ref{W-EQ}. For, it will be crucial to have a computable bound for $d$ in the above proposition. 
So we will pay special attention to the relation between the $d$ produced by the lemma and the following constant:
\begin{definition}
Let $f(\overline{X},Y)\in \Zp\{\overline{X},Y\}$. Then $S(f)$ denotes a constant so that for all $\overline{x}\in \mathcal{O}_p$ either $f(\overline{x},Y)$ is identically zero or has less than $S(f)$ roots in $\mathcal{O}_p$ (counting multiplicities).
\end{definition}
In particular, we will see that the $d$ we obtained can be bounded by a constant $K(f)$ that depends only on $S(f), S(g_1),\cdots, S(g_l)$ for some functions $g_1,\cdots, g_l$ constructed from $f$ (in the Weiertrass system). So, in the above proposition, we can take $d=K(f)$. Indeed, let $d$ obtained from the proposition and $d'>d$. Take ${b'}_{IJ}=0$ if $|I|\geq d'$ and $d'>|J|\geq d$ and ${b'}_{IJ}=b_{IJ}$ if $|I|\geq d'$ and $|J|<d$. Then the series  ${b'}_{IJ}$ satisfies all requirement of the first part of the proposition. So we may assume $d=K(f)$.

\begin{claim}\label{c-claim 1}
Let $a_1,\cdots, a_n\in \Zp$ with $v(a_i)=v(a_j)$ for all $i,j$. Then, there is $t\in \mathcal{O}_p$ with $v(t)=0$ such that $v(\sum a_it^i)=v(a_j)$.
\end{claim}
\begin{proof}
Without loss of generality, we may assume that $v(a_i)=0$. Assume that the claim is false i.e.\ that for all $t\in \mathcal{O}_p$, $v(\sum a_it^i)>0$. Then, $\sum Res(a_i)t^i=0$ for all $t\in \Fp^{alg}$. So, $Res(a_i)=0$ which contradicts the assumption that $v(a_i)=0$.
\end{proof}
\begin{claim}\label{c-claim 2}
If $f(\overline{X})\in\Zp\{\overline{X}\}$, $\|f\|=1$. Then for all $\overline{x}\in\Zp^n$ with $v(\overline{x})=0$ there is $\overline{t}\in \mathcal{O}_p^n$ with $v(\overline{t})=1$ such that $v(f(\overline{xt}))=0$.
\end{claim}
\begin{proof}
Let $f(\overline{X})=\sum a_I\overline{X}^I$ with $\|f\|=1$. As $f\in\Zp\{\overline{X}\}$ there is a finite set $\mathcal{K}$ such that $v(a_I)=0$ for all $I\in \mathcal{K}$ and $v(a_I)>0$ for all $I\notin \mathcal{K}$. Let $\overline{x}\in\Zp^n$ such that $v(\overline{x})=0$. Then by claim \ref{c-claim 1}, there is $\overline{t}\in \mathcal{O}_p^n$ with $v(\overline{t})=0$ such that
$$v\left( \sum_{I\in\mathcal{K}} a_I\overline{tx}^I\right) = v(a_I\overline{tx}^I)=v(a_I)=0. $$
So, $v(f(\overline{tx}))=0$.
\end{proof}

\par First, we start with the special case of Proposition \ref{Weierstrass Denef-vdD lemma 1.4} where no parameters are involved i.e.\ $|\overline{X}|=0$. In the next lemma, the functions $b_{IJ}$'s are given by the coefficients of the functions $g_J$'s (in that case the series $b_{IJ}$ are just constants in $\Zp$).
\begin{lemma}\label{c-lemma 1}
Let $f(\overline{Y})=\sum a_I\overline{Y}^I\in \Zp\Wsyst{\overline{Y}}$. Then, there is $A(f)$ such that for all $J$ such that $|J|<A(f)$, there is $g_J\in \Zp\Wsyst{\overline{Y}}$ with
$$f=\sum_{|J|<A(f)} a_J\overline{Y}^J(1+pg_J(\overline{Y})).$$
Furthermore, $A(f)$ can be bounded in terms of $S(f), S(h_1),\cdots, S(h_l)$ where $h_i$ is obtained from $f$ by derivation or by composition with a polynomial.
\end{lemma}

\begin{proof}
By induction on $n:=|\overline{Y}|$:
\par If $n=1$: $f=\sum_i a_iY^i$. Then, there is $D$ so that for all $i>D$ and $j<D$, $v(a_D)\leq v(a_j)$ and $v(a_D)<v(a_i)$. So,
$$f=a_0+a_1Y+\cdots + a_DY^D(1+\frac{a_{D+1}}{a_D}Y+\cdots).$$
And, by definition of $D$, $v\left(\frac{a_{D+k}}{a_D}\right)>0$ for all $k>0$. Take $A(f)=D+1$,
$$g_i=0 \mbox{ for $i<D$, }g_D= \frac{f-\sum_{i\leq D} a_iY^i}{pa_DY^D}.$$
Note that by Strassmann's theorem (which states that, for $D$ as defined above, $f$ has $D$ roots in $\mathcal{O}_p$ (counting multiplicities), see \cite{Robert} section VI 2.1 for instance), $A(f)\leq S(f)+1$.
\par If $f=\sum a_I\overline{Y}^I$. Let $I$ so that for all $J$ such that $j_k>i_k$ for some $k$ and $j_l\geq i_l$ for all $l$, we have $v(a_I)<v(a_J)$ and $v(a_I)$ is minimal among the $v(a_K)$ (we can pick $I$ so that $v(a_I)$ is minimal and $|I|$ is maximal for this property). Then,
$$\sum_{J, j_l\geq i_l} a_J\overline{Y}^J=a_I\overline{Y}^I\left(1+\sum_{K\in \N^n} \frac{a_{K+I}}{a_I}\overline{Y}^K\right).$$
Let $pg_I=\sum_{K\in \N^n} \frac{a_{K+I}}{a_I}\overline{Y}^K$.
\par For all $k\leq n$, $s<i_k$. By $J_{k,s}$ we denote an index in $\N^n$ whose $k$th coordinate is $s$. Then, by closure under derivation,
$\sum_{J_{k,s}} a_{J_{k,s}} \overline{Y}^{J_{k,s}}\in \Zp\Wsyst{\overline{Y}}$. So, by induction,
$$f_{k,s}=\sum_{J\in \N^{n-1}} a_{J_{k,s}} \overline{Y'}^{J}=\sum_{|J|<A(f_{s,k})} a_{J_{k,s}}\overline{Y'}^{J} (1+h_J(\overline{Y'})).$$
Take $g_{J_{k,s}}=h_J$. Set $A(f)=\max\{ |I|+1, A(f_{s,k})\}$, then rearranging the series, one can obtain series in $\Zp\Wsyst{\overline{Y}}$ which satisfy the properties of the lemma.
\par To prove that $A(f)$ is bounded in terms of $S(f), S(h_1),\cdots, S(h_l)$, it is sufficient by induction to give an upper bound for $|I|$. Let $g(\overline{Y},\overline{T})=f(\overline{T\cdot Y})$. Then, by claim \ref{c-claim 1} there is $\overline{t}$ such that $$v\left( \sum_{|J|=|I|}a_J\overline{t}^J\right) = v(a_I).$$
Let $\widetilde{g}(Z,\overline{T})=g(Z,\cdots, Z, \overline{T})$. By Strassmann's theorem, $\widetilde{g}(Z,\overline{t})$ has $|I|$ roots in $\mathcal{O}_p$. So, $|I|\leq S(\widetilde{g})$.
\end{proof}

\begin{lemma}\label{c-lemma 2}
Let $f(\overline{X},\overline{Y})=\sum a_I(\overline{X})\overline{Y}^I\in \Zp\Wsyst{\overline{X},\overline{Y}}$. Then, there is $B(f)$ such that for all $J$ with $|J|<B(f)$, there is  $g_{J}\in\Zp\Wsyst{\overline{X},\overline{Y}}$ so that 
$$f=\sum_{|J|<B(f)}a_J(\overline{X})g_J(\overline{X},\overline{Y}).$$
Furthermore, $B(f)$ is bounded in terms of constants $A(g)$ (see Lemma \ref{c-lemma 1}) where $g$ is obtained from $f$ using Weierstrass divisions.
\end{lemma}
The existence part is Lemma 4.2.14 in \cite{Cluckers-Lipshitz}. We recall here the proof in order to show the second statement.
\begin{proof}
Let $f(\overline{X},\overline{Y})=\sum_J a_J(\overline{X})\overline{Y}^J=\sum_{IJ} a_{IJ}\overline{X}^I\overline{Y}^J$. Then by Lemma \ref{c-lemma 1}, 
$$f=\sum_{|IJ|<A(f)} a_{IJ}\overline{X}^I\overline{Y}^J(1+g_{IJ}(\overline{X},\overline{Y})). $$
Let $I_0J_0$ maximal (for the lexicographic order) so that $v(a_{I_0J_0})$ is minimal among the $v(a_{KL})$. Then, $a_{I_0J_0}^{-1}f$ is preregular of order $I_0J_0$. To keep the notation simpler, we will now replace $f$ by $a_{I_0J_0}^{-1}f$. We have that $a_{J_0}(\overline{X})$ is preregular of order $I_0$. So, after a bijective change of variables (as in the proof of Theorem \ref{W-EQ}), $a_{I_0}(\overline{Z})$ is regular of order $s$ in $Z_n$. By Weierstrass division,
$$f=a_{I_0}(\overline{Z})U(\overline{Z},\overline{Y}) + (R_0(\overline{Z'},\overline{Y})+\cdots+ R_{s-1}(\overline{Z'},\overline{Y})Z_n^{s-1}),$$
where $U,R_i\in \Zp\Wsyst{\overline{Z},\overline{Y}}$. Let $R_i=\sum_JR_{iJ}(\overline{Z'})\overline{Y}^J$, $R=\sum_i R_iZ_n^i$ and $R_J=\sum_i R_{iJ}Z_n^i$. Then, by induction on $|\overline{Z}|$ (if $|\overline{Z}|=0$, we can take $B(R_0)=A(R_0)$), we have that
$$R_0=\sum_{|I|<B(R_0)}R_{0I}g_I.$$ Also, identifying the coefficients, one see that $a_I=a_{I_0}U_I+R_I$. Then,
\begin{align*}
 f-a_{I_0}U& -\sum_{|I|<B(R_0)} R_I g_I\\
 &=R-R_0-\sum_{|I|<B(R_0)}\sum_{0<i<s} R_{iJ}g_I\\
		& = \sum_{0<i<s}Z_n\Big(R_i-\sum_{|I|<B(R_0)}R_{iI}g_I\Big)\\
		&=: Z_n \sum_I S_I(\overline{Z})\overline{Y}^J=Z_n S(\overline{Z},\overline{Y}).
\end{align*}
where $S$ is a polynomial in $Z_n$. Then as
$$S_J=a_J+a_{I_0}U_J-\sum_{|I|<B(R_0)} R_Ig_{IJ} $$
we can conclude the proof of the lemma by induction on $|\overline{Z}|$ and $s$.
\par It remains to prove that $B(f)$ is bounded by constants $A$ obtained in \ref{c-lemma 1}. By the proof of Lemma \ref{c-lemma 1}, we see that we can assume $|I_0J_0|=A(f)-1$. So, $s\leq i_{0n}+\cdots + i_{01}A(f)^{n-1}\leq nA(f)^n$. $B(R_0)$ is bounded by $A(R_0)$ in the case $|\overline{Z}|=0$ otherwise we proceed by induction. Similarly for $B(S(\overline{Z},\overline{Y}))$. Then, $B(f)\leq \max\{ A(f), B(R_0), B(S)\}$.
\end{proof}

\begin{prop}\label{c-prop 1}
Let $f(\overline{X},\overline{Y})\in \Zp\Wsyst{\overline{X},\overline{Y}}$. Then, there is $C(f)$ and for all $|J|<C(f)$, there is  $g_{J}\in\Zp\Wsyst{\overline{X},\overline{Y}}$ so that 
$$f=\sum_{|J|<C(f)}a_J\overline{Y}^J(1+pg_J).$$
Furthermore, $C(f)$ is bounded in terms of $B(f), S(f)$ and $S(h_1),\cdots, S(h_l)$ where $h_i$ is constructed from $f$.
\end{prop}
The existence is proved in \cite{Cluckers-Lipshitz} Theorem 4.2.15. Again, we go through their proof in order to prove the second statement.
\begin{proof}
Let $\|(a_1,\cdots, a_n)\|=\max\{|a_i|\}$.
By Lemma \ref{c-lemma 2}, $f=\sum_{|I|<B(f)}a_I(\overline{X})g_I(\overline{X},\overline{Y})$. Let $g_{I}(\overline{X},\overline{Y})=\sum_{J}g_{IJ}(\overline{X})\overline{Y}^J$.  For $I,J\in \{0,\cdots, B(f)\}^m$, we can assume that $g_{IJ}=1$ if $I=J$ and $0$ otherwise. Indeed, let 
$$\widetilde{g}_I=\left(g_I-\sum_{\|J\|< B(f)} g_{IJ}\overline{Y}^J+\overline{Y}^I\right),$$
if $|I|<B(f)$ and $\widetilde{g}_I=\overline{Y}^I$ if $|I|\geq B(f), \|I\|<B(f)$. Then,
\begin{align*}
\sum_{\|I\|<B(f)} a_I\widetilde{g}_I&= \sum_{|I|<B(f)} a_I\left(\overline{Y}^I+\sum_{\|J\|\geq B(f)} g_{IJ}\overline{Y}^I\right) + \sum_{|I|\geq B(f), \|I\|<B(f)} a_I \overline{Y}^I\\
&= \sum_{|I|<B(f)} a_I\overline{Y}^I + \sum_{|I|<B(f)}\sum_{\|J\|\geq B(f)} a_Ig_{IJ}\overline{Y}^J + \sum_{|I|\geq B(f), \|I\|<B(f)} a_I \overline{Y}^I\\
&= \sum_{ \|I\|<B(f)} a_I \overline{Y}^I+ \sum_{|I|<B(f)}\sum_{\|J\|\geq B(f)} a_Ig_{IJ}\overline{Y}^J\\
&= \sum_{ \|I\|<B(f)} a_I \overline{Y}^I+ \sum_{\|J\|\geq B(f)}\sum_{|I|<B(f)} a_Ig_{IJ}\overline{Y}^J\\
&= \sum_{ \|I\|<B(f)} a_I\overline{Y}^I + \sum_{\|J\|\geq B(f)} a_J\overline{Y}^J=f.
\end{align*}
For $M(f)$ large enough, we can furthermore assume that $\|g_{IJ}\|<1$ for all $\|IJ\|\geq M(f)$.
\par Let us start the proof of the last part of the proposition: assume $M(f)$ minimal such that it satisfies the property. Then, if $M(f)\not=0$, there is $I_0$ with $|I_0|=M(f)-1$ and $\|a_{I_0}\|=1$. So by claim \ref{c-claim 2}, for all $\overline{x}\in \Zp^n$ with $v(\overline{x})=0$ there is $\overline{t}\in \mathcal{O}_p^m$ such that $v(a_{I_0}(\overline{tx}))=0$. By claim \ref{c-claim 1}, there is $\overline{u}\in \mathcal{O}_p^n$ such that 
$$v\left(\sum_{|I|=|I_0|} a_I(\overline{tx})\right)=v(a_{I_0}).$$
Let $g(\overline{X},\overline{T},\overline{U},Z):=f(\overline{XT}, \overline{U}Z)$. Then by Strassmann's theorem and claim \ref{c-claim 1} there is $\overline{u}$ such that $g(\overline{x},\overline{t},\overline{u}, Z)= f(\overline{xt},\overline{u}Z)$ has $|I_0|=M(f)-1$ root in $\mathcal{O}_p$. So, $M(f)\leq S(g)+1$. 
\par Let $\mathcal{K}=\{1,\cdots , M(f)\}^m\cup \{I\mid\ i_k\geq M(f) \mbox{ for all }k\}$. Then (if we take $M(f)$ also bigger than $B(f)$),
$$f_{\mathcal{K}}=\sum_{I\in\mathcal{K}}a_I\overline{Y}^I = \sum_{I\in \{0,\dots, M(f)\}^m} a_I\overline{Y}^I(1+h_I(\overline{X},\overline{Y}))$$
where $h_I=\widetilde{g_I}-\overline{Y}^I\in \Zp\Wsyst{\overline{X},\overline{Y}}$. Then, take $C(f_{\mathcal{K}})=M(f)$.
\par $f-f_{\mathcal{K}}$ is handled by induction: like in Lemma \ref{c-lemma 1}, we consider $f_{s,k}=\sum_{I_{k,s}} a_{I_{k,s}}\overline{Y'}Y_k^s$. We use the inductive hypothesis to show that $f_{s,k}$ satisfies the proposition. Then, the proposition is proved as $f=f_{\mathcal{K}}+\sum_{s< M(f),k\leq m}f_{s,k}$ : take $C(f) = \max\{M(f), C(f_{s,k})\}$.
\end{proof}
Then in the proof of Theorem \ref{W-EQ} and Proposition \ref{Weierstrass Denef-vdD lemma 1.4} , one can take 
$$\widetilde{f}:=\sum_{J<I} V_J\overline{X}^I(1+pg_J)+\overline{X}^I+\sum_{I<J, |J|<d}pV_J\overline{X}^I(1+pg_J)$$
 where the $g_J$ are given by Proposition \ref{c-prop 1}. The function $b_{IJ}$ are determined by the coefficients of $pg_J$: $pg_J(\overline{X},\overline{Y})=\sum b_{IJ}(\overline{Y})\overline{X}^I$ (i.e.\ $b_{IJ}$ is a derivative of $pg_J$ evaluated at zero) and $d$ is bounded by $C(f)$. This completes the proof of Proposition \ref{Weierstrass Denef-vdD lemma 1.4} and of Theorem \ref{W-EQ}.

\section{Weierstrass system generated by a set of restricted analytic functions}\label{Weierstrass system generated by}

\par Let $F$ be a family of restricted analytic functions. As before, we denote by $\mathcal{L}_F$ the expansion of the language $\mathcal{L}_{Mac}$ by the elements of $F$. We will prove that under the condition that the set of $\mathcal{L}_F$-terms is closed under derivation and \emph{decomposition functions} (to be defined later), the theory $\Z_{p,F}$ is strongly model-complete.
\par Let $W$ be any Weierstrass system which contains $F$. Then the theory of $\Zp$ eliminates the quantifiers in the language $\mathcal{L}_W^D$. In particular, if the functions in $W$ are $\mathcal{L}_F$-existentially definable, we are done. In this section, we will define a Weierstrass system $W_F$ such that any function in $W_F$ is constructible from the data set $F$ i.e.\ for all $f\in W_F$, there exists a finite collection of functions $f_1,\cdots , f_k\in F$ from which one can construct $f$ using polynomial combinations, Weiestrass divisions, permutations of the variables and inverses. We will see in the next section that under the above assumptions on $F$, any function in $W_F$ is actually existentially definable.

\par We define \emph{the Weierstrass system generated by the $\mathcal{L}_F$-terms} by:
\par For each $n$, let $W_{F,n}^{(0)}$ be the set of $\mathcal{L}_F$-terms with $n$ variables.
We define $W_{F,n}^{(m+1)}$ by induction on $m$. Assume that we have defined $W_{F,n}^{(k)}$ for each $n\in \N$ and for each $k\leq m$. Then, $W_{F,n}^{(m+1)}$ is the ring generated by:
\begin{enumerate}[(a)]
	\item $W_{F,n}^{(m)}\subset W_{F,n}^{(m+1)}$;
	\item For all $f\in W_{F,n}^{(m)}$, for all permutations $\sigma$, $f(X_{\sigma(1)},\cdots , X_{\sigma(n)})\in W_{F,n}^{(m+1)}$;
	\item For all $f\in W_{F,n}^{(m)}$, if $f$ is invertible in $\Zp\{\overline{X}\}$, then $f^{-1}\in W_{F,n}^{(m+1)}$;
	\item For all $f,g\in W_{F,n}^{(m)}$, if $f$ is divisible by $g(0)$ in $\Zp\{\overline{X}\}$, then $f/g(0)\in W_{F,n}^{(m+1)}$;
	\item For each $f\in W_{F,n+1}^{(m)}$ regular of order $d$ in $X_{n+1}$, for each $g\in W_{F,n+1}^{(m)}$, the functions $A_0, \cdots , A_{d-1}\in \Zp\{X_1,\cdots , X_{n}\}$ and $Q\in \Zp\{X_1,\cdots , X_{n+1}\}$ given by the Weierstrass division and their partial derivatives belong to $W_{F,n}^{(m+1)}$ and $W_{F,n+1}^{(m+1)}$ respectively.
\end{enumerate}
Let $W_{F,n} := \bigcup_m W_{F,n}^{(m)}$. It is clear that these sets determine a Weierstrass system over $\Zp$. We denote this system by $W_F$. Then, by Theorem \ref{W-EQ}, the theory of $\mathbb{Z}_{p}$  admits elimination of quantifiers in $\mathcal{L}_{W_F}^D$. We will show that each function of $W_F$ is strongly definable in $\mathcal{L}_F$ (under extra assumptions on $F$).

\par Note that by definition, for all $f\in W_{F,n}^{(m+1)}$, there exist $g_1,\cdots, g_k\in W_{F,n+1}^{(m)}$ such that $f$ is obtained from $g_1,\cdots , g_k$ using the above operations (a)-(e) and polynomial combinations. We denote this property by $f\in \langle g_1,\cdots , g_k\rangle$. We denote $f\in \langle f_1,\cdots, f_k\rangle^*$ if we have a family of functions $f_{i,j}$ ($1\leq i\leq k$, $1\leq j \leq n$) such that
\begin{itemize}
	\item  $f\in W_F^{(m+n)}, f_{i,j}\in W_F^{(m+n-j )}$ for all $i, j$;
	\item $f_{1,n}=f_1,\cdots, f_{k,n}=f_k$;
	\item $f\in \langle f_{1,1}, \cdots ,  f_{k,1}\rangle$ and $f_{ij}\in\langle f_{1, j+1},\cdots , f_{k,j+1}\rangle$ for all $i, j$.
\end{itemize}
It should be clear that by induction one can find for each $f\in W_F$ a finite collection of $\mathcal{L}_F$-terms $f_1,\cdots, f_d$ such that $f\in \langle f_1,\cdots , f_d\rangle^*$. Furthermore,
\begin{lemma}\label{lemma 3.1}Let $\Psi(\overline{X})\equiv \exists Y_1,\cdots , Y_n\phi(\overline{X},\overline{Y})$ be a $\mathcal{L}_F$-formula where $\phi$ is quantifier-free that is a boolean combination of formulae of the form $f(\overline{X},\overline{Y})=0$ or $P_n(g(\overline{X},\overline{Y}))$.
Then, there exists $\phi'$ a quantifier-free $\mathcal{L}_{W_F}^D$-formula such that
$$\Zp\vDash \forall \overline{X}\Big(\Psi(\overline{X})\rightleftarrow \exists Z_1, \cdots , Z_{n-1} \phi'(\overline{X},\overline{Z})\Big).$$
Furthermore, for any subterm $f$ in $\phi'$ (not involving $D$), there exists a subterm $h$ in $\phi$ and $P_1,\cdots , P_m$ polynomials with coefficients in $\Z$ such that $f\in \langle h, P_1,\cdots , P_m\rangle^*$ 
\end{lemma}
This follows immediately from the proof of Theorem \ref{W-EQ}. And, by induction, there exists a quantifier-free $\mathcal{L}_{W_F}^D$-formula $\varphi(\overline{X})$ equivalent to $\Psi$ such that for any term $f$ in $\phi$, $f\in \langle g_1,\cdots , g_l, P_1,\cdots ,P_s\rangle^*$ where $g_1,\cdots , g_l$ are the $\mathcal{L}_F$-subterms in $\Psi$ and $P_1,\cdots ,P_s$ are polynomials with coefficients in $\Z$.

\section{Decomposition functions and definability of finite algebraic extensions}\label{Decomposition functions}

\par Let $F$ be a family of restricted analytic functions and $W_F$ be the Weierstrass system generated by the $\mathcal{L}_F$-terms. We want to prove that any function of $W_F$ is $\mathcal{L}_F$-existentially definable. Let $f\in W_F^{(m+1)}$. Then there are $g_1,\cdots, g_k\in W_F^{(m)}$ such that $f\in \langle g_1,\cdots , g_k\rangle$. Assume that each function $g_i$ is existentially definable. Then so is $f$ if it is constructed from $g_1,\cdots, g_k$ using the operations (a)-(d) and polynomial combinations. However, it is not clear whether it is also the case when $f$ is obtained using Weierstrass division. In general, we couldn't conclude that this is the case. So, we will add extra-conditions on $F$ so that the functions involved in the Weierstrass division are existentially definable from the data set.
First, we illustrate the main idea of the existential definition on a simple example:
\par Let $f$ be a $\mathcal{L}_F$-term regular of order $d$ in $X_{n+1}$. Then, by the Weierstrass preparation theorem, there are $A_0,\cdots ,A_{d-1}\in W_{F,n}^{(1)}$ and a unit $U\in 
W_{F,n+1}^{(1)}$ such that:
$$ f(X_1,\cdots , X_{n+1}) = \Big[X_{n+1}^d + A_{d-1}(\overline{X'})X_{n+1}^{d-1} + \cdots +A_0(\overline{X'})\Big]\cdot U(\overline{X}),$$
where $\overline{X'}=(X_1,\cdots , X_n)$. We want to give an existential definition of the functions $A_0,\cdots , A_{d-1}, U$. 
\par Fix $\overline{x'} = (x_1,\cdots , x_n)\in \Zp^n$. It is rather clear that $U(\overline{x'},X)$ is strongly definable in terms of $f$  and $A_0(\overline{x'}),\cdots , A_{d-1}(\overline{x'})$. The graph of $U$ is determined by the graph of $f, A_{d-1}, \cdots, A_0$ (roughly, $U=f/(X_{n+1}^d+\cdots+A_0)$; we refer to \cite{vdDries4} Lemma 3.4 for the precise definition).
\par Let $\alpha_1,\cdots , \alpha_d$ be the roots of $P(X) := \sum A_i(\overline{x'})X^i+ X^d$ in $\Qp^{alg}$. Note that these are exactly the roots of $f(\overline{x'},X)$ in $\Qp^{alg}$ with nonnegative valuation. Then, the coefficients $A_i(\overline{x'})$ are uniquely determined by $\alpha_1,\cdots , \alpha_d$. For instance, if the roots are nonsingular (i.e.\ if $\alpha_i\not= \alpha_j$ for all $i\not= j$), the coefficients $A_i(\overline{x'})$ are uniquely determined by the system:
$$T(\overline{\alpha}, A_0(\overline{x'}),\cdots, A_{d-1}(\overline{x'}))\equiv
\left(\begin{array}{cccc}
1 & \alpha_1 &\cdots & \alpha_1^{d-1}\\
\vdots & \vdots && \vdots \\ 
1 & \alpha_d &\cdots &\alpha_d^{d-1}
\end{array}\right)
\cdot
\left(\begin{array}{c}
A_0(\overline{x'})\\
\vdots\\
A_{d-1}(\overline{x'})
\end{array}\right) =
\left(\begin{array}{c}
\alpha_1^d\\
\vdots\\
\alpha_d^d
\end{array}\right).
$$
Other similar systems determine the coefficients in the case where the roots are singular. The above relation leads to an existential formula which determines the graphs of the functions $A_i$: $(\overline{x'},\overline{a})\in \mbox{Graph}(A_0,\cdots, A_{d-1})$ iff the formula
 $$\Psi (\overline{x'},\overline{a})\equiv \exists \overline{\alpha}\in \Qp^{alg}\ \bigwedge_i f(\overline{x'},\alpha_i)=0\wedge\Big(\bigwedge_{i\not=j}\alpha_i\not=\alpha_j 
 \wedge \bigwedge_i v(\alpha_i)\geq 0 \wedge T(\overline{\alpha},\overline{a})\Big) \bigvee \Big[\cdots\Big]$$
  is satisfied in $\Zp$, where $\Big[\cdots\Big]$ holds for the disjunction of the systems determining $A_0(\overline{x'}),\cdots , A_{d-1}(\overline{x'})$ in all possible singular cases.
 However, the existential quantifiers in this formula quantify over elements in $\Qp^{alg}$ (the $\alpha_i$'s). Actually, using properties of finite extension of $\Qp$ we can replace $\Qp^{alg}$ by a finite algebraic extension in the above formula:

\par It follows from Krasner's lemma that the $p$-adic field $\Qp$ has finitely many algebraic extensions of a given degree (which can be assumed generated by elements algebraic over $\Q$). So, we can construct a sequence of finite algebraic extensions $K_1\subseteq K_2 \subseteq \cdots$ such that:
\begin{itemize}
\item $K_n$ is the splitting field of $Q_n(X)$ polynomial of degree $N_n$ with coefficients in $\Q$;
\item $K_n= \Qp(\beta_n)$ for all $\beta_n$ root of $Q_n$;
\item any extension of degree $n$ is contained in $K_n$ and its valuation ring is contained in $V_n:=\Zp[\beta_n]$.
\end{itemize}
Let us remark that for all $\overline{x'}\in \Zp^n$, $\alpha_1,\cdots , \alpha_d\in V_d$. So, in the above formula $\Psi$, we can quantify over $V_d$ instead of $\Qp^{alg}$.
\par Let $f\in F$. Then, $f$ defines an analytic function on $V_d$. So, we can consider the structure $(V_d,+, -, \cdot ,0,1, P_n, f;\ n\in \N, f\in F)$. If this structure is existentially definable in $\mathbb{Z}_{p,F}$ then the above formula $\Psi$ can be translated in $\Zp$ and we are done.\\
It is well known that the structure of ring is definable but this may not be the case for the elements of $F$. We will extend $F$ by a family of functions  $\widetilde{F}$ so that the structure $(V_d,+, -, \cdot ,0,1, P_n, f;\ n\in \N, f\in \widetilde{F})$ is existentially definable in $\mathbb{Z}_{p,\widetilde{F}}$.
\par For this, it is sufficient to describe the decomposition of $f$ in the basis of $V_d$ over $\Zp$. Fix $f\in F$ and $y= \sum y_i\beta_d^i\in V_d^k$ (where $y_i\in \Zp^k$). We decompose $f(y)$ in the basis of $V_d$ over $\Zp$:
$$ f(y) =f\left(\sum y_i\beta_d^i\right) = c_{0,f,d}(\overline{y}) + c_{1,f,d}(\overline{y})\beta_d + \cdots + c_{N_d-1,f,d} (\overline{y})\beta_d^{N_d-1},\qquad (*)$$
where $\overline{y}=(y_0,\cdots, y_{N_d-1})$. It determines functions $c_{i,f,d}\in \Zp\{X_1,\cdots, X_{kN_d}\}.$ We call these functions \emph{the decomposition functions of $f$ in $K_d$}. Note that these functions are independent of the choice of $\beta_d$. Indeed, for all $\sigma$ in the Galois group of $K_d$ over $\Qp$ (denoted by $Gal(K_d/\Qp)$), 
$$ f(y^\sigma) =f\left(\sum y_i{\beta_d^\sigma}^i\right) = c_{0,f,d}(\overline{y}) + c_{1,f,d}(\overline{y}){\beta_d^\sigma} + \cdots + c_{N_d-1,f,d} (\overline{y}){\beta_n^\sigma}^{N_d-1}, (**)$$
by continuity of $\sigma$. Let $\widetilde{F} := F\cup \{c_{i,f,d}\mid\ f\in F,\ d\in \N \mbox{ and } i< N_d\}$. Then, by definition,
\begin{lemma}
For all $d$, the structure $(V_d,+, -, \cdot ,0,1, P_n, f;\ n\in \N, f\in F)$ is existentially definable in $\mathbb{Z}_{p,\widetilde{F}}$.
\end{lemma}
At this point, one may expect that we will need to add further decomposition functions so that the structure $(V_d,+, -, \cdot ,0,1, P_n, f;\ n\in \N, f\in \widetilde{F})$ is also definable. However this is not the necessary. Indeed, let us remark that the $c_{i,f,d}(\overline{y})$ are linear combinations of the $f(y^\sigma)$: by $(**)$,
$$
\left(\begin{array}{c}
c_{0,f,d}(\overline{y})\\
\vdots\\
c_{N_d-1,f,d}(\overline{y})
\end{array}\right)
=
V^{-1}
\left(\begin{array}{c}f(y^{\sigma_1})\\
\vdots\\
f(y^{\sigma_{N_d}})
\end{array}\right),
$$
where $V$ is the Vandermonde matrix of the roots of $Q_d$ and $\sigma_i$ are the elements of $Gal(K_d/\Qp)$. So, as power series,
$$det\ V\cdot c_{i,f,d}(\overline{y})=\sum a_i\beta _d^i f\left(\sum R_i(\overline{y})\beta_d^i\right), $$
where $a_i\in \Q\cap \Zp$ and $R_i$ is a polynomial with coefficients in $\Zp\cap \Q$. Therefore, the above relation holds for all $\overline{y}\in V_l^{kN_d}$. So,

\begin{prop}\label{interpretation of V_d}
For all $d$, the structure $(V_d,+, -, \cdot ,0,1, P_n,f;\ n\in \N, f\in\widetilde{F})$ is existentially definable in $\mathbb{Z}_{p,\widetilde{F}}$.
\end{prop}

Finally note that if the set of $\mathcal{L}_F$-terms is closed under derivation, so is the set of $\mathcal{L}_{\widetilde{F}}$-terms. This follows immediately from the above equality $(*)$.

\section{Strong model-completeness}\label{Strong model-completeness}

\par First, let us describe the existential definitions of the functions in $W_F$.
\begin{prop}\label{strong def W-syst}
Let $F$ be a family of functions in $\Zp\{\overline{X}\}$. Assume that the set of $\mathcal{L}_F$-terms is closed under derivation. Let $\widetilde{F}$ be the extension of $F$ by the decomposition functions in $K_d$ of each $f\in F$ (for all $d\in \N$). Let $g\in W_{\widetilde{F}}$. Then $g$ is strongly definable in $\mathcal{L}_{\widetilde{F}}$. Furthermore, for all $d$, the structure $(V_d,+, -, \cdot, 0, 1, g)$ is strongly definable in $\mathbb{Z}_{p,\widetilde{F}}$.
\end{prop}
Given a function $f\in \Zp\{X_1,\cdots ,X_n\}$, we denote the set $\left\{\frac{\partial^k f}{\partial X_i^k}; 1\leq i \leq n, k\in\N \right\}$ by $[f]$.
\begin{proof}
The proof is very similar to the corresponding results in \cite{vdDries4}. The existential definitions given below are the $p$-adic equivalent of the real case.
\par Let us recall that for all $f\in W_{\widetilde{F},n}^{(m+1)}$, there exist $g_1,\cdots , g_k\in W_{\widetilde{F}}^{(m)}$ such that $f\in \langle g_1,\cdots , g_k\rangle$.
So, it is sufficient to prove by induction on $m$ that
\begin{enumerate}
	\item For all $f\in W_{\widetilde{F},n}^{(m+1)}$, $f$ and its derivatives are strongly definable in terms of functions in $W_{\widetilde{F},n+1}^{(m)}$ (and their derivatives);
	\item The definitions remain true uniformly over the algebraic extensions $V_d$ i.e.\ the graphs of the function $f:V_d^k\rightarrow V_d$ and of its derivatives are strongly definable in terms of functions in $W_{\widetilde{F},n+1}^{(m)}$ (and their derivatives). 
\end{enumerate} 
By definition of the language $\mathcal{L}_{\widetilde{F}}$ and by Proposition \ref{interpretation of V_d}, it is clear that the extensions of the functions in $W_{\widetilde{F},n}^{(0)}$ to $V_d$ are definable. And so are the graphs of their derivatives as the set of $\mathcal{L}_{\widetilde{F}}$-terms is closed under derivation.  So, we assume by induction that the graph of the extension to $V_d$ of any function in $W_{\widetilde{F},n}^{(k)}$ (or one of its derivative) is strongly definable in our structure for all $d$, for all $n$ and for all $k\leq m$.
\par Let $f\in W_{\widetilde{F},n}^{(m+1)}$. Then, $f=P(f_1,\cdots , f_k)$ where $P\in \Z[\overline{Y}]$ and $f_1,\cdots , f_k\in W_{\widetilde{F},n}^{(m+1)}$ are functions of the type (a)-(e) in the definition of Weierstrass system generated by the $\mathcal{L}_F$-terms. If the functions $f_1,\cdots, f_k$ satisfy properties 1. and 2., then $f$ also satisfies these properties.
Indeed, the graph of $f$ is strongly definable in terms of $f_1,\cdots, f_k$ as $(\overline{x},y)$ is a point of the graph of $f$ as functions from $\Zp$ to itself (or as function from $V_d$ to itself if the below formula is satisfied in $V_d$) iff

$$\Zp\vDash \exists t_1\cdots \exists t_k \bigwedge t_i=f_i(\overline{x})\wedge y=P(t_1,\cdots , t_k). $$
Similarly for the derivatives of $f$.
So, we can assume that $f$ is a function of the type (a)-(e).
\par The cases where $f$ is obtained as the division of a function $g\in W_{\widetilde{F},n}^{(m)}$ by division by $h(0)$ or is a function $g$ in $W_{\widetilde{F},n}^{(m)}$ (i.e.\ $h=1$) are obvious: $(\overline{x},y)\in Graph(f)$ iff
$$\Zp \vDash h(0)y=g(\overline{x}). $$
If $f(\overline{X})=g(X_{\sigma(1)},\cdots, X_{\sigma(n)})$ where $\sigma$ is a permutation of $\{1,\cdots , n\}$ then the tuple $(\overline{x},y)$ belongs to the graph of $f$ iff 
$$ \Zp\vDash  \exists \overline{t}\ \bigwedge_i t_i=x_{\sigma(i)} \wedge y=g(\overline{t}).$$
If $f$ is the inverse of a function $g$, then $(\overline{x},y)$ belongs to the graph of $f$ iff
$$\Zp\vDash yg(\overline{x})=1. $$
Therefore, in these cases (a)-(d), both the graphs of $f$, of its derivatives and their extensions to $V_d$ are strongly definable in terms of $[g]$.
So, we are reduced to the case (e):
\par Let $f,g\in W_{\widetilde{F},n+1}^{(m)}$ where $f$ has order $d$ in $Y=X_{n+1}$. Then, there are $A_0,\cdots , A_{d-1}\in W_{\widetilde{F},n}^{(m+1)}$ and $Q \in W_{\widetilde{F},n+1}^{(m+1)}$ such that
$$ g= Qf+ \Big( A_{d-1}Y^{d-1} + \cdots  + A_1Y + A_0\Big).$$
We have to prove that  $A_0,\cdots , A_{d-1}, Q$ (and their derivatives) are strongly definable in $\Zp$ and that the definitions work uniformly over the algebraic extensions $V_d$.
\begin{Fact}\label{strong def W-coeff} $A_0,\cdots, A_{d-1}$ are strongly definable in terms of $[f,g]$.
\end{Fact}
\begin{proof}
Fix $\overline{x}\in  \Zp^{n}$.
Let $\alpha_1,\cdots , \alpha_d$ be the roots of $f(\overline{x},Y)$ in $V_d$ (we take in account multiplicities). Then, $A_0(\overline{x}),\cdots , A_{d-1}(\overline{x})$ are uniquely determined by these roots.
Indeed, first assume that the roots are distinct. In this case, $A_0(\overline{x}),\cdots , A_{d-1}(\overline{x})$ are determined by the relations:
$$\alpha_i\not=\alpha_j \mbox{ for all }i,j$$
$$f(\overline{x},\alpha_i)=0 \mbox{ for all }i$$
$$\left(\begin{array}{cccc}
1 & \alpha_1 &\cdots & \alpha_1^{d-1}\\
\vdots & \vdots && \vdots \\ 
1 & \alpha_d &\cdots & \alpha_d^{d-1}
\end{array}\right)
\left(\begin{array}{c}
A_0(\overline{x})\\
\vdots\\
A_{d-1}(\overline{x})
\end{array}\right) =
\left(\begin{array}{c}
g(\overline{x},\alpha_1)\\
\vdots\\
g(\overline{x},\alpha_d)
\end{array}\right).$$

If $f(\overline{x},Y)$ admits singular roots, say $\alpha_1=\alpha_2$ and $\alpha_i\not= \alpha_j$ for all $i\not=j$, $i,j\not=2$ for instance, then we replace the $d$ equations $f(\overline{x},\alpha_1)=\cdots = f(\overline{x}, \alpha_d)=0$ by $f(\overline{x}, \alpha_1)=\frac{\partial f}{\partial Y}(\overline{x}, \alpha_1)=f(\overline{x},\alpha_3)=\cdots = f(\overline{x}, \alpha_d)=0$. The functions $A_i$ are determined in this case by the relations:
$$\alpha_i\not=\alpha_j \mbox{ for all }i\not= j,\ j\not=2$$
$$f(\overline{x},\alpha_i)=0 \mbox{ for all }i\not=2$$
$$\frac{\partial f}{\partial Y}(\overline{x}, \alpha_1)=0$$
$$\left(\begin{array}{cccc}
1 & \alpha_1 &\cdots & \alpha_1^{d-1}\\
0 & 1 &\cdots & (d-1)\alpha_1^{d-2}\\
1 & \alpha_3 &\cdots & \alpha_3^{d-1}\\
\vdots & \vdots && \vdots \\ 
1 & \alpha_d &\cdots & \alpha_d^{d-1}
\end{array}\right)
\left(\begin{array}{c}
A_0(\overline{x})\\
A_1(\overline{x})\\
A_2(\overline{x})\\
\vdots\\
A_{d-1}(\overline{x})
\end{array}\right) =
\left(\begin{array}{c}
g(\overline{x},\alpha_1)\\
\frac{\partial g}{\partial Y}(\overline{x},\alpha_1)\\
g(\overline{x},\alpha_3)\\
\vdots\\
g(\overline{x},\alpha_d)
\end{array}\right).$$
For each configuration of multiplicities of the roots of $f(\overline{x},Y)$, the coefficients $A_i$ are completely determined by a system like above. We proceed to a disjunction over all possible cases to define the graphs of $A_0,\cdots , A_{d-1}$ on $\Zp^{n}$.
\par Let $\Psi(\overline{x},A_0(\overline{x}),\cdots, A_{d-1}(\overline{x}),\overline{\alpha})$ be the disjunction of all possible systems like above. Then, the following formula gives an existential definition of the graphs of $A_0,\cdots , A_{d-1}$:
$$\exists \alpha_1\cdots \alpha_d\in V_d\ \Psi(\overline{x},A_0(\overline{x}),\cdots, A_{d-1}(\overline{x}),\overline{\alpha}). $$
Let us remark that the above definition is an existential definitions where we quantify over $V_d$. We interpret this formulas in $\Zp$. So, formally, the $\alpha_i$'s are replaced by tuples. The additions, multiplications (in $V_d$ in $\Psi$) are replaced by their interpretation in $\Z_p$. Similarly, the functions $f,g$, their derivatives are also replaced by their interpretations in $\Zp$ (which exists by inductive hypothesis).
\par Note also that the $\alpha_i$'s are only unique up to permutation. It means that so far, we have only existentially defined the graphs of the $A_i$'s. This is a consequence of the existence of Skolem function in $\Qp$: We transform this existential definition into a strong existential formula using \cite{Denef}. In this paper, J. Denef gives a formula of definable selection for finite sets i.e.\ a quantifier-free formula $D (x,X)$ (where $X$ is a new predicate) such that for all $X(\overline{v})$ a predicate corresponding to a finite set in $\Qp$:
$$
\begin{array}{rl}
\Qp\vDash \exists v_1,\cdots, v_s &\Big[ \bigwedge_i X(v_i)\wedge \bigwedge_{i,j} v_i\not=v_j\Big]\\
& \rightarrow \exists !v_1,\cdots \exists ! v_s \Big[ \bigwedge_i X(v_i) \wedge \bigwedge_iD(v_i,X)\wedge \bigwedge_{i,j} v_i\not=v_j \Big].
\end{array}$$
We use this formula with $X$ equals to the set $\{\alpha_1,\cdots ,\alpha_d\}$ (interpreted in $\Zp$) to get a strong definition of the graphs of the $A_i$'s.
\end{proof}
Note that the above formula works uniformly over the algebraic extensions. Therefore, the graphs of the $A_i's$ as functions from $V_d^n$ to $V_d$ are also strongly definable.

\begin{Fact}[Lemma 3.4 \cite{vdDries4}]\label{strong def W-unit}  $Q$ and its derivatives (with respect to $Y$) are strongly definable in terms of $[f,g], A_0,\cdots, A_{d-1}$.
\end{Fact}
\begin{Fact}[Proposition 3.8 \cite{vdDries4}]\label{strong def W-deriv1}  For all $I, j$, $\frac{\partial^I A_0}{\partial \overline{X}^I},\cdots , \frac{\partial^I A_{d-1}}{\partial \overline{X}^I}$ and $\frac{\partial^I\ \ \partial^j Q}{\partial\overline{X}^I\partial Y^j}$ are strongly definable in terms of $[f,g], A_0,\cdots, A_{d-1}, Q$.
\end{Fact}
One can adapt these formula in the $p$-adic context as above. Again, it leads to existential definitions where the quantifiers are over $V_d$ and we have to interpret these formulas in $\Zp$. Note that the definitions also work uniformly over finite algebraic extensions.

\par This proves that $A_0, \cdots , A_{d-1},Q$ and their derivatives are strongly definable functions in terms of functions in $W_{\widetilde{F},n+1}^{(m)}$ and therefore completes the proof of the proposition.

\end{proof}
The first main theorem follows immediately from Theorem \ref{W-EQ} and Proposition \ref{strong def W-syst}
\begin{theorem}\label{strong model-completeness}
Let $F$ be a family of restricted analytic functions. Assume that the set of $\mathcal{L}_F$-terms is closed under derivation. Let $\widetilde{F}$ be the extension of $F$ by the decomposition functions of $f$ for each $f\in F$. Then, $\mathbb{Z}_{p,\widetilde{F}}$ is strongly model-complete in $\mathcal{L}_{\widetilde{F}}$.
\end{theorem}

\section{Effective Weierstrass system}\label{Effective Weierstrass system}

\par We are now interested in an effective version of Theorem \ref{strong model-completeness} i.e.\ is there an algorithm which takes for entry a $\mathcal{L}_{\widetilde{F}}$-formula and return an existential $\mathcal{L}_{\widetilde{F}}$-formula equivalent to it. First, we remark that we need some way to encode formulas. This is only possible if $F$ is countable. Therefore, from now on, we will assume that this is the case. We fix a G\"odel numbering for the language $\mathcal{L}_{\widetilde{F}}$. Then every term and formula has a code attached to it. In Theorem \ref{strong model-completeness}, we assume that the set of $\mathcal{L}_F$-terms is closed under derivation. It is important that the derivation is effective i.e.\ that for all variable $X_i$, there is an algorithm which takes for entry (code for) a $\mathcal{L}_F$-term (say $f(\overline{X})$) and return (the code for) the $\mathcal{L}_F$-term $\frac{\partial f}{\partial X_i}$. Whenever all these hypotheses are satisfied we will say that $F$ is an \emph{effective family of restricted analytic functions}. Note that we do not need to assume that the set of $\mathcal{L}_{\widetilde{F}}$-terms is closed under (effective) derivation: by definition of the decomposition functions, this set is closed under derivation and we have an explicit formula for the derivation of the decomposition functions in terms of the elements of $F$.

\par Let $F$ be an effective family of restricted analytic functions and $W_{\widetilde{F}}$ be the Weierstrass system generated by the $\mathcal{L}_{\widetilde{F}}$-terms. Then for each $f\in W_{\widetilde{F}}$ there exists an existential $\mathcal{L}_{\widetilde{F}}$-formula that defines the graph of $f$: this is the statement of Proposition \ref{strong def W-syst}. Let $\mathcal{E}_F$ be the set of existential $\mathcal{L}_{\widetilde{F}}$-formulas in the form of Proposition \ref{strong def W-syst}. By the proof of the proposition, this set of formulas is recursively enumerable. We also fix some recursive rule so that our formula is written in such way that we keep track of each step of the inductive procedure and of the different cases that are used i.e.\ given a code for a function we have to be able to reconstruct how that function has been obtained (using operations (a)-(e) in the definition of Weierstrass system generated by the $\mathcal{L}_F$-terms).
 Each element of this set has a code attached to it and given a code for an $\mathcal{L}_{\widetilde{F}}$-formula, we can determined whether or not this code correspond to an element of $\mathcal{E}_F$. Note that it is possible that an element of $\mathcal{E}_F$ does not interpret in $\Zp$ the graph of an element of $W_F$ (for instance, if we write a formula that is the definition of the coefficient of Weierstrass division applied to a nonregular function). This is not an issue: indeed, in the course of the proof of model-completeness, we only use elements in $\mathcal{E}_F$ that are constructed as the graph of an element of $W_F$ (and a code for such an element is computable from the way it is constructed). In an other direction, the same element of $W_F$ can be coded by two different elements of $\mathcal{E}_F$: again it is not an issue as we are only interested in the model-completeness. The full decidability of the theory would require to determine whether or not two such codes (i.e.\ existential formulas) interpret the graph of the same function.

\par Let us recall Theorem \ref{W-EQ}: the theory of $\Zp$ in the language of Macintyre expanded by symbols for the element of a Weierstrass system and division admits the elimination of quantifiers.
Assume that the  Weierstrass system is of the type $W_{\widetilde{F}}$ for some $F$ effective family of restricted analytic functions. Then, an inspection of the proof given in \ref{W-EQ} shows that it is recursive except for the use of Proposition \ref{Weierstrass Denef-vdD lemma 1.4} and the construction of $\widetilde{f}$. Let $\Phi$ be a $\mathcal{L}_{\widetilde{F}}$-formula. Once we have determined a procedure to compute $d(f)$ and a code for $\widetilde{f}$ (from the code $f$), any function that appear in the proof of quantifier elimination is in $W_{\widetilde{F}}$ and we can compute a code attached to it. So, we can code a quantifier-free equivalent to $\phi$ (in $\mathcal{L}_{W_{\widetilde{F}}}$). Then by Proposition \ref{strong def W-syst}, we can compute an existential $\mathcal{L}_{\widetilde{F}}$-formula equivalent to $\Phi$.
\par In Lemma \ref{c-lemma 1}, Lemma \ref{c-lemma 2} and Proposition \ref{c-prop 1}, we give an explicit description of the construction of $\widetilde{f}$ and of the functions $b_{IJ}$. Furthermore, the $d$ that appear in Proposition \ref{Weierstrass Denef-vdD lemma 1.4} can be bounded in terms of $S(f)$ and $S(g_1),\cdots, S(g_l)$ where $g_i\in W_{\widetilde{F}}$ is obtained from $f$ in an explicit way i.e.\ there is a computable element of $\mathcal{E}_F$ attached to each $g_i$. So, if $S(g)$ is computable from a code for $g$, the proof of Theorem \ref{W-EQ} can be done recursively and therefore, we obtain effective model-completeness. We will say that a Weierstrass system generated by $\mathcal{L}_{\widetilde{F}}$-terms is effective if the constants $S(g)$ are computable for all $g\in W_{\widetilde{F}}$:
\begin{definition}
A Weierstrass system $W_{\widetilde{F}}$ is called \emph{effective} if there exists an algorithm which takes for entry $e$ a code for an element of $\mathcal{E}_F$ and return an integer $S(e)$ such that, if $e$ is the code for a function interpreted in $\Zp$ by the graph of an element $f(\overline{X},Y)$ in $W_{\widetilde{F}}$, for all $\overline{x}\in\Zp$ the set
$$\{y\in \mathcal{O}_p\mid\ f(\overline{x},y)=0\}$$
is either infinite or has cardinality less than $S(e)$.
\end{definition}
Note for $f\in W_{\widetilde{F}}$ that appear in the proof of Theorem \ref{W-EQ} we have a code $e\in W_{\widetilde{F}}$ attached to it. We set $S(f)=S(e)$. This is a slight abuse of notation as it could happen that some $f'$ that is construct in some other way is equal to $f$. In that case, we have a code $e'$ attached to $f'$ and it could happen that $S(e)\not=S(e')$. This is not an issue in our case as we are only interested by model-completness. So, we have no need to check whether or not two existentential formulas that define the graph of a function interprete the same function in $\Zp$ (so, there is not harm to assume this is not the case). 
\par  If $W_{\widetilde{F}}$ the Weierstrass system generated by the $\mathcal{L}_{\widetilde{F}}$-terms is effective then the strong model-completeness in Theorem \ref{strong model-completeness} is effective. In fact, under some extra-hypotheses on $F$, it will turn out that it is sufficient to have control on the $\mathcal{L}_F$-terms to obtain an effective Weierstrass system. 
\begin{definition} Let $F$ be an effective family of restricted analytic functions. We say that $F$ has an effective Weierstrass bound if there is an algorithm which takes for entry a code for a $\mathcal{L}_{\widetilde{F}}$-term $f$ and return an integer $d(f)$ which satisfies the properties of Fact \ref{Denef-vdD lemma 1.4 bis} i.e.\ that if $f(\overline{X},\overline{Y})=\sum a_I(\overline{X})\overline{Y}^I$, there are $b_{IJ}(\overline{X})\in \Zp\Wsyst{\overline{X}}$ with $\|b_{IJ}(\overline{X})\|<1$ and $\|b_{IJ}\|\rightarrow 0$ as $|I|\rightarrow \infty$, such that 
for all $I$ with $|I|\geq d(f)$,
$$a_I(\overline{X}) = \sum_{|J|<d(f)}b_{IJ}(\overline{X}) a_J(\overline{X}).$$

\end{definition}
Note that in the above definition, $d$ depends on a choice of a partition $(\overline{X},\overline{Y})$ on the variables of $f$. In fact, in the definition, either we need an algorithm for each possible partitions or we may also take $d(f)$ to be the max of all $d_{(\overline{X},\overline{Y})}(f)$ over all possible partition (i.e.\ we may assume that $d(f)$ is independent of a particular choice of partition). Let us remark that it is not required that the series $b_{IJ}$ are (effectively) existentially definable. We explain now the link between $S(f)$ for $f\in W_{\widetilde{F}}$ and $d(g)$ for $\mathcal{L}_{\widetilde{F}}$-term $g$. 
Note that in the above definition, $d$ depends on a choice of a partition $(\overline{X},\overline{Y})$ on the variables of $f$. In fact, in the definition, either we need an algorithm for each possible partitions or we may also take $d(f)$ to be the max of all $d_{(\overline{X},\overline{Y})}(f)$ over all possible partition (i.e.\ we may assume that $d(f)$ is independent of a particular choice of partition). Let us remark that it is not required that the series $b_{IJ}$ are (effectively) existentially definable. We explain now the link between $S(f)$ for $f\in W_{\widetilde{F}}$ and $d(g)$ for $\mathcal{L}_{\widetilde{F}}$-term $g$. 
\par  Let $f$ be an $\mathcal{L}_{\widetilde{F}}$-term and $d(f)$ as defined above. Then, $S(f)\leq d(f)$. Indeed, $S(f)$ is a uniform bound on the the number of solutions in $\mathcal{O}_p$ of $f(X,\overline{y})=0$. Strassmann's theorem states that it is equals to the index of the last coefficient of minimal valuation. By definition, $d(f)$ is an upper bound for the index of this coefficient. 
\par In general, let $f$ be a function in our Weierstrass system. Then, there are integers $n$ and $m+1$ such that $f\in W_{{\widetilde{F}},n}^{(m+1)}$. The function $f$ has an existential definition in terms of functions in $W_{{\widetilde{F}},n+1}^{(m)}$:  there exist $g_1,\cdots , g_k\in W_{{\widetilde{F}},n+1}^{(m)}$ such that $f\in \langle g_1,\cdots , g_k\rangle$. We will see that $S(f)$ can be bounded in terms of $d(g_1),\cdots d(g_k)$ (actually; one also need to take in account derivatives of these functions and composition with polynomials).
Going down by induction, we may assume that the $g_i$'s are $\mathcal{L}_{\widetilde{F}}$-terms. So, assuming that $d(g_1),\cdots, d(g_k)$ are computable for all $g_i$ $\mathcal{L}_{\widetilde{F}}$-terms (i.e.\ $F$ has an effective Weierstrass bound), we will be able to compute $S(f)$.
\par The cases where $f$ is obtained from a function $g$ by inversion, permutation of the variables or division by a constant are rather easy (the number of zeros of $f$ is immediately determined by the number of zeros of $g$). The main difficulty is the case where $f$ is obtained using Weierstrass division. In this case, by the definitions given in the facts \ref{strong def W-coeff} to \ref{strong def W-deriv1} in Proposition \ref{strong def W-syst}, we see that zeros of such a function correspond to zeros of systems of $n'$ equations in $W_{F,n'}^{(m)}$ (with the same parameters as the one that appear in $f$).

\par In section \ref{Effective bound}, we will bound the number of solutions in $(\mathcal{O}_p)^n$ of a general system of $n$ analytic functions with $n$ variables (uniformly over parameters) in an effective way (depending on the constants $d(f)$ for any $f$ in the system or one of its derivatives). For this, we will use results of tropical analytic geometry from \cite{Rabinoff}. These results relate the number of solutions of the system to a geometric volume. This volume will be in turn bounded effectively in terms of $d(g)$'s where $g$ is any function in the system or one of its derivatives. However, to use the results we may need to apply a small perturbation to our system as follow:
\begin{definition}
Let $f\in\Zp\{\overline{X}\}$. We say that $f$ is overconvergent if there is a ball $B$ around zero that strictly contains $\mathcal{O}_p^n$ and such that $f$ converges on $B$.
\end{definition}
Let $f\in \Zp\{\overline{X}\}$ overconvergent. Let $t\in \mathcal{O}_p$ with positive valuation less than $\varepsilon$. Let $\widehat{f}=f(t^{-1}X)$. If $f$ is overconvergent, then $\widehat{f}\in \Zp\{\overline{X}\}$ if $\varepsilon$ is small enough (in fact, $\widehat{f}\in \Qp\{\overline{X}\}$ but as we are interested in the zeros of $f$, we can multiply by a scalar (that only depends on $\varepsilon$ and not on $t$) so that all coefficients are in $\Zp$). As $\widehat{f}$ is a restricted power series, we may defined $d(\widehat{f})$ and $S(\widehat{f})$ as before. To get effective model-completeness, it will be required that $d(\widehat{f})$ is computable (for all term $f$ in our language). 
\par Let us give some more precise definition of our setting. The key example to keep in main is the case of exponential terms.
We consider $F$ such that each $\mathcal{L}_F$-term is overconvergent (the radius of convergence may be different for each term). We assume that $F$ is closed under decomposition functions and that the set of $\mathcal{L}_F$-terms is closed under derivation in an effective way (in the sense of the beginning of the section). We will finally assume that $F$ satisfies the following definition:
\begin{definition}
We say that the set of $\mathcal{L}_{\widetilde{F}}$-terms has an effective generalised Weierstrass bound if there is an algorithm which takes for entry a code for a $\mathcal{L}_{\widetilde{F}}$-term and return an integer $d'(f)$ such that there exists $\varepsilon(f)>0$ such that for all $t\in \mathcal{O}_p$ with positive valuation less than $\varepsilon(f)$ , $d(\widehat{f})\leq d'(f)$ where $\widehat{f}$ is obtained from $f$ after the change of variable $X_i\rightarrow X_it^{-1}$ for all $i$. 
\end{definition}
Note that in the above definition, $\varepsilon(f)$ may depend on $f$ but we do not require that the $\varepsilon(f)$ is computable. The existence of $\varepsilon(f)$ for the interpretation of the language is sufficient. This existence will be guaranteed by the hypothesis of overconvergence. In fact, the algorithm takes for entry $\mathcal{L}_{\widetilde{F}}$-term i.e.\ a syntaxic object and return $d'(f)$ so that for the interpretation of the language there is some $\varepsilon$ such that for all $t$ with $0<v(t)<\varepsilon$, $d(\widehat{f})\leq d'(f)$. In our application in the exponential case, we will take  $d'(f)$ to be $d(\widehat{f})$ for some $t$ with computable valuation (so in that case, $\varepsilon$ is computable).
The above definition makes sense as we are interested in change of variables with $v(t)$ small enough and because of the following inequalities:
\begin{lemma}
 Let $t,t'\in \mathcal{O}_p$ with $0<v(t)<v(t')$. Let $\widehat{f}$ (resp. $\widehat{f}'$) be the series obtained from $f$ after the change of variables $X_i\rightarrow X_it^{-1}$ (resp. $X_i\rightarrow X_it^{-1}$). Then, $d(\widehat{f})\leq d(\widehat{f}')$ (and, $S(\widehat{f})\leq S(\widehat{f'})$).
\end{lemma}
\begin{proof}
Let $f(\overline{X},\overline{Y})=\sum_I a_I(\overline{X})\overline{Y}^I$. Then,
$$\widehat{f}(\overline{X},\overline{Y})=\sum a_I(\overline{X}) \frac{\overline{Y}^I}{t^{|I|}}=:\sum \widehat{a}_I(\overline{X})\overline{Y}^I;$$
$$\widehat{f'}(\overline{X},\overline{Y})=\sum a_I(\overline{X}) \frac{\overline{Y}^I}{{t'}^{|I|}}=:\sum \widehat{a}_I'(\overline{X})\overline{Y}^I.$$
As $v(t')>v(t)$, $t'=ct$ for some $c$ with $v(c)>0$. Assume that for all $I$, $|I|\geq d(\widehat{f}')$, there are $b_{IJ}$ such that $\|b_{IJ}\|\leq 1$, $\|b_{IJ}\|\rightarrow 0$ as $|I|\rightarrow \infty$ and for all $I$ with $|I|\geq d(\widehat{f}')$
$$\widehat{a}_I'=\sum_{|J|<d(\widehat{f'})} \widehat{a}_J' b_{IJ},$$
By definition of $\widehat{a}_I'$ and as $t'=ct$, it implies that
$$\frac{a_I}{{t'}^{|I|}}=\frac{a_I}{t^{|I|}c^{|I|}}=\sum_{|J|<d(\widehat{f'})} \frac{a_J}{t^{|J|}c^{|J|}} b_{IJ}$$
i.e.\
$$\frac{a_I}{t^{|I|}}=\widehat{a}_I=\sum_{|J|<d(\widehat{f'})} \widehat{a}_J \frac{c^{|I|}}{c^{|J|}} b_{IJ}.$$
As $|I|>|J|$, $v(\frac{c^{|I|}}{c^{|J|}} )>0$. Set ${b'}_{IJ}:= \frac{c^{|I|}}{c^{|J|}}b_{IJ}$. Then for all $I$, $|I|\geq d(\widehat{f}')$, $\|{b'}_{IJ}\|\leq 1$, $\|{b'}_{IJ}\|\rightarrow 0$ as $|I|\rightarrow \infty$ and for all $I$ with $|I|\geq d(\widehat{f}')$
$$\widehat{a}_I=\sum_{|J|<d(\widehat{f'})} \widehat{a}_J {b'}_{IJ}.$$
So, $d(\widehat{f})\leq d(\widehat{f}')$.
\end{proof}

Our second main theorem is that assuming $F$ to be an effective family of restricted overconvergent analytic functions and that the set of $\mathcal{L}_{\widetilde{F}}$-terms has an effective generalised Weierstrass bound, then $W_{\widetilde{F}}$ is an effective Weierstrass system and so we have effective model-completeness. But first, we prove the promised result on the effective bound for system of analytic functions.

\section{Effective bound on the number of solutions in $\mathcal{O}_p$ of some effective analytic system}\label{Effective bound}

\par First, we start this section by stating some results and definitions from \cite{Rabinoff} that will be used in our proofs. Let us remark that we do not state the definitions nor the results in full generality but we have restricted them in the case of our interest. In particular, the results hold if we replace $\Qp$ by any of its finite algebraic extension or by $\Cp$.

\par Let $P=\prod [r_i,\infty)\subset \R^n$ with $r_i\in \Q$. Then, $\Zp\langle P\rangle$ denotes the set of power series in $\Zp[[\overline{X}]]$ convergent on the product of balls with center $0$ and radius $p^{-r_i}$ i.e.\

$$\Zp\langle P\rangle = \left\{\sum a_I\overline{X}^I\mid\ v(a_I)+ \langle I, v(\overline{x})\rangle\rightarrow \infty\ \forall \overline{x}\mbox{ such that }v(\overline{x})\in P\right\},$$
(where $\langle \cdot, \cdot \rangle$ denotes the usual scalar product and the limit is taken over $|I|\rightarrow \infty$). For instance, if $P=\prod [0,\infty)^n$ then $\Zp\langle P\rangle=\Zp\{\overline{X}\}$.
\par Let $\overline{x}\in \Cp^n$. The \emph{tropicalization of $\overline{x}$}, denoted by $trop(\overline{x})$, is the tuple formed by the valuations of the $x_i$'s:
$$trop(\overline{x})=(v(x_1),\cdots , v(x_n)). $$ 
Let $f\in \Zp\langle P\rangle$ and $C\subseteq \overline{P}:=\prod [r_i,\infty]$, we denote
$$V(f;C)=\{\overline{x}\in\Cp\mid\ trop(\overline{x})\in C \mbox{ and } f(\overline{x})=0\}. $$
If $C=\overline{P}$, we denote the above set by $V(f)$.\\
We define the \emph{tropicalization of $f$} as the closure of the set
$$\{\nu\in \overline{P}\mid \mbox{  there exists $\overline{x}\in V(f)$ such that and } trop(\overline{x})=\nu\},$$
where the closure is taken in $\overline{P}$. We denote this set by $Trop(f, P)$ or by $Trop(f)$ when $P$ is clear from the context. Similarly, if $Y$ is a subset of $\Cp^n$, $Trop(Y)$ denote the image of $Y$ by the map $trop$ in $(\R\cup\{\infty\})^n$.
\par $Trop(f)$ is actually completely determined by the coefficients of $f$:
Let $f=\sum a_I\overline{X}^I\in \Zp\langle P\rangle$. Fix $\nu\in P$. Let
\begin{align*}
vert_\nu(f) =\{(I,v(a_I))\mid\ &v(a_I)+\langle I,\nu\rangle\leq val(a_{I'})+\langle I',\nu\rangle \\
																					 &\mbox{ for all monomials $a_{I'}\overline{X}^{I'}$ of }f\}.
\end{align*}
This is the set of points such that the valuation of the monomial $a_I\overline{x}^I$ is minimal (among all valuation of the monomial of the series) for $trop(\overline{x})=\nu$. As $f\in \Zp\langle P\rangle$, $v(a_I)+\langle I,\nu\rangle\rightarrow \infty$. So, $vert_\nu(f)$ is actually a finite set. Furthermore, it is proved in \cite{Rabinoff} (Lemma 8.2) that $vert_P(f)=\bigcup_{\nu\in P} vert_\nu(P)$ is finite.
\par We define the \emph{initial form of $f$ with respect to $\nu$} to be
$$ in_\nu(f)=\sum_{(I,v(a_I))\in vert_\nu(f)}a_I\overline{X}^I\in \Zp[\overline{X}].$$
Let us remark that
$$vert_\nu(f)=\{(I,v(a_I))\mid\ a_I\overline{X}^I\mbox{ is a monomial of } in_\nu(f)\}.$$

Let $f\in \Zp\langle P\rangle$. Let $\overline{t}\in \Cp^n$ such that $f(\overline{t})=0$. By the ultrametric inequality, we have that for some $I,I'\in \N^n$ distinct, $v(a_I\overline{t}^I)=val(a_{I'}\overline{t}^{I'})=\min_J\{val(a_J\overline{t}^J)\}$. So, if $\nu=v(\overline{t})\in Trop(f)$, $inv_\nu(f)$ is not a monomial. A crucial result in \cite{Rabinoff} is that the converse is true:
\begin{lemma}[Lemma 8.4 in \cite{Rabinoff}] Let $f\in \Zp\langle P\rangle$ nonzero. Then,
$$Trop(f)=\{\nu\in \overline{P}\mid\ inv_\nu(f)\mbox{ is not a monomial}\}.$$
\end{lemma} 
So, $Trop(f)$ is determined by $inv_\nu(f)$ i.e.\ by the coefficients of $f$.
$Trop(f)\cap \R^n$ is actually a rather simple subset of $\R^n$ : a polyhedral complex.
\begin{definition}
A \emph{polyhedron} is a finite intersection of half-hyperplane in $\R^n$. The dimension of a polyhedron $P$ is the dimension of the smallest affine subspace of $\R^n$ containing $P$.
We refer to \cite{Rabinoff} section 2 for the formal definitions of faces and other notion from convex geometry.
A \emph{polyhedral complex} is a finite collection $\Pi$ of polyhedra in $\R^n$ (called \emph{faces} or \emph{cells} of $\Pi$) such that 
\begin{itemize}
	\item if $P,P'\in \Pi$, $P\cap P'\not=\varnothing$, then $P\cap P'$ is a face of $P$ and a face of $P'$;
	\item for all $P\in \Pi$ if $F$ is a face of $P$ then $F\in \Pi$.
\end{itemize}
The \emph{support} of $\Pi$, denoted $\lvert\Pi\rvert$ is the set $\bigcup_{P\in \Pi}P$. The dimension of $\Pi$ is the dimension of the highest dimensional cell of $\Pi$.
\end{definition}
For $\nu\in Trop(f)\cap \R^n$, we define
$$\gamma_\nu = \{\nu'\in Trop(f)\cap \R^n\mid\ vert_{\nu'}(f)\supseteq vert_\nu(f)\}.$$
If $Trop(f)$ is non-empty and $f$ nonzero, the collection $\{\gamma_\nu, \nu\in Trop(f)\cap \R^n\}$ is a polyhedral complex in $\R^n$ of codimension at least $1$ (i.e.\ all maximal cells have dimension at most $n-1$). The support of this complex is exactly $Trop(f)\cap \R^n$. We will denote by $Trop(f)\cap \R^n$ the complex as well as its support.
\par Let $\pi: \N^n\times \R\longrightarrow \N^n$ denote the projection on the $n$ first coordinates. We define
$$ \check{\gamma}_\nu = \pi(conv(vert_\nu(f)));$$
where $conv()$ denotes the convex closure of the set in $\R^n$.
This a bounded polyhedron. The \emph{Newton complex of $f$} is the collection of polyhedra $\{\check{\gamma}_\nu\mid\ \nu\in P\}$. We denote by $New(f, P)$ this set or by $New(f)$ when $P$ is clear from the context. In general this set is not a polyhedral complex: some face of a polyhedron in $New(f)$ may not belong to $New(f)$. Indeed, a 
 face of a polyhedron $\check{\gamma}_\nu$ may correspond to the projection of a set $conv(vert_\nu(f))$ where $\nu\notin P$ (or $f$ is not convergent at elements of
  tropicalization $\nu$). It turns out that it is a polyhedral complex in the case where $f$ is polynomial (in which case we consider the set of all $\check{\gamma}_\nu$ for $\nu\in \R^n$). The support of $New(f)$ is 
$$\lvert New(f)\rvert =conv\{I\in \N\mid (I,val(a_I))\in vert_\nu(f)\mbox{ for some }\nu\in Trop(f)\cap \R^n\}.$$
We will also denote this support by $New(f)$. The complexes $New(f)$ and  $Trop(f)\cap \R^n$ are dual to each other in the following sense:
\begin{prop}[J. Rabinoff \cite{Rabinoff} Proposition 8.6.2]\label{Rabinoff prop 8.6.2}
\begin{enumerate}
	\item For all $\nu,\nu'\in Trop(f)\cap \R^n$, $\gamma_\nu$ is a face of $\gamma_{\nu'}$ iff $\check{\gamma}_{\nu'}$ is a face of $\check{\gamma_{\nu}}$.
	\item For all $\nu\in  Trop(f)\cap \R^n$, $\gamma_\nu$ and $\check{\gamma_\nu}$ are orthogonal in the sense that the linear subspaces of $\R^n$ associated to the affine spans of $\gamma_\nu$ and $\check{\gamma}_\nu$ are orthogonal. Furthermore, $dim(\gamma_\nu)+dim(\check{\gamma}_\nu)=dim(\R^n)$.
\end{enumerate}
\end{prop}
The above proposition implies that we have one-to-one correspondence between cells of $Trop(f)\cap \R^n$ and positive dimensional polyhedra in $New(f)$.
\begin{Exp}\label{Newton polydrehon 2}
 Let $f(x,y)=px+x^p+y^p$. We have drawn the tropicalization and the Newton polygon of $f$ in figure \ref{figure}.
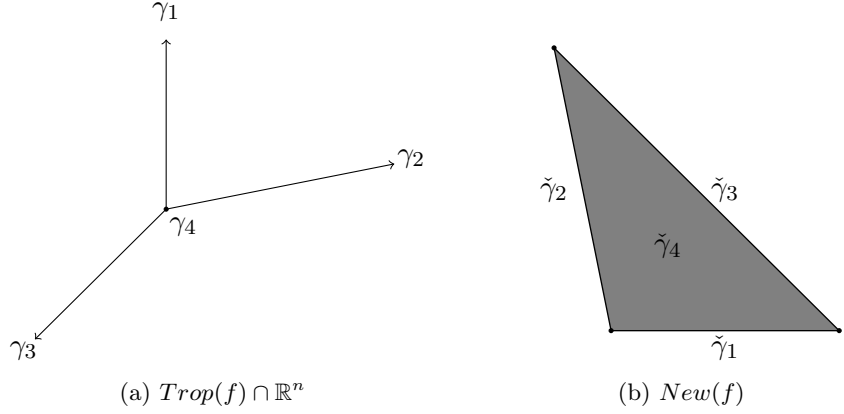
\begin{figure}
\centering
\subfloat[$Trop(f)\cap \R^n$]{
\begin{tikzpicture}[scale=0.75]\label{figure1}
   \draw[->] (2.5,2.5)--(0.2,0.2);
   \node at (0.0,0.0) {$\gamma_3$};
   \draw[->] (2.5,2.5)--(2.5,5.5);
   \node at (2.5,6) {$\gamma_1$};
   \draw[->] (2.5,2.5)--(6.5,3.3);
   \node at (6.8,3.4) {$\gamma_2$};
   \draw[fill=black](2.5,2.5) circle (1.0pt);
   \node at (2.8,2.2) {$\gamma_4$};
\end{tikzpicture}}
\hspace{1cm}
\subfloat[$New(f)$]{
\begin{tikzpicture}[scale=0.75]\label{figure2}

		\path[fill=gray] (1.0,2.0) -- (5.0,2.0)  -- (0.0,7.0) -- cycle;
    \draw[fill=black](1.0,2.0) circle (1.0pt);	
    \draw[fill=black](5.0,2.0) circle (1.0pt);
    \draw[fill=black](0.0,7.0) circle (1.0pt);
    \draw[line width=0.5] (1.0,2.0)-- (5.0,2.0);
		\draw[line width=0.5] (1.0,2.0)-- (0.0,7.0);
		\draw[line width=0.5] (0.0,7.0)-- (5.0,2.0);
		\node at (3.0,1.7) {$\check{\gamma}_1$};
		\node at (0.0,4.5) {$\check{\gamma}_2$};
		\node at (3.0,4.5) {$\check{\gamma}_3$};
		\node at (2.0,3.5) {$\check{\gamma}_4$};
\end{tikzpicture}}
\caption{The tropicalization and Newton complex of $px+x^p+y^p$.}
    \label{figure}
\end{figure}
Where in these figures, we take $P=(-\infty,+\infty)^2$ (with the obvious extensions of the definitions). If $P=[r,\infty)\times[s,\infty)$, then $Trop(f)\cap \R^n$ is the intersection between the set described in the above figure and $P$. $New(f)$ is the collection of all $\check{\gamma_i}$ such that $\gamma_i\cap P$ has the same dimension that $\gamma_i$.
\end{Exp}

One of the main result of \cite{Rabinoff} is a generalization of the classical result on Newton polygons for power series in $\Zp$. It relates the number of solutions  with a given valuation to the (mixed) volume of some polyhedron in $New(f)$. First let us define the notion of mixed volume:
\begin{definition} Let $P_1,\cdots , P_n$ be bounded polyhedra in $\R^n$. The \emph{Minkowsky sum of $P_1,\cdots ,P_n$} is 
$$ P_1+\cdots+P_n = \{v_1+\cdots +v_n\mid\ v_i\in P_i\}.$$
For $\lambda\in \R_{\geq 0}$, we set $\lambda P_i=\{\lambda v\mid\ v\in P_i\}$. We define the function 
$$\begin{array}{rll}
	V_{P_1\cdots P_n}: \R_{\geq 0}^n & \longrightarrow & \R \\
	(\lambda_1,\cdots , \lambda_n) &\longmapsto & vol(\lambda_1P_1+\cdots +\lambda_nP_n)
\end{array}$$
where $vol$ is the usual Euclidean volume. The function $V_{P_1\cdots P_n}$ is actually a homogeneous polynomial in $\lambda_1\cdots \lambda_n$ of degree $n$. The \emph{mixed volume} $MV(P_1\cdots P_n$) is defined to be the coefficient of the $\lambda_1\cdots\lambda_n$-term of $V_{P_1\cdots P_n}$.
\end{definition}
\begin{Remark} The function $MV$ is monotonic. So, if $P_1,\cdots, P_n\subset P$,
$$MV(P_1,\cdots , P_n)\leq MV(P,\cdots , P)=Vol(P).$$
\end{Remark}
\begin{theorem}[J. Rabinoff \cite{Rabinoff} Theorem 11.7]\label{Rabinoff 11.7} Let $f_1,\cdots , f_n\in \Zp\langle P\rangle$. Then for all $\nu \in \bigcap_i Trop(f_i)\cap \R^n$ isolated in the interior of $P$, let $\check{\gamma_i}=\pi (vert_\nu(f_i))\in New(f_i)$. Then
$$ \left|  \bigcap_iV(f_i;\{\nu\})\right|\leq MV(\check{\gamma_1},\cdots , \check{\gamma_n}). $$
\end{theorem}

We fix now $F$ a family of restricted analytic functions like in section \ref{Effective Weierstrass system} i.e.\ such that the set of $\mathcal{L}_F$-terms is closed under derivation (in a effective way). We will also assume that the set of  $\mathcal{L}_{F}$-terms has an effective generalised Weierstrass bound. Let $W_F^{(0)}$ denote the set of $\mathcal{L}_{F}$-terms.
\par We will now prove that if $f_1,\cdots ,f_n\in W_F^{(0)}$ then uniformly over the parameters $\overline{y}$, we can compute a bound on the number of isolated points in $\bigcap Trop(f_i)$ and on the number of zeros  in $\bigcap_i V(f_i)$ with tropicalization $\nu$  (for a fixed isolated valuation $\nu$ in $\bigcap Trop(f_i)\cap\R_{>0}^n$, the bound will not depend on the choice of $\nu$). Actually, these number will be bound by a recursive function depending on some $d(g)$'s where the $g$'s could be $f_i$ or some of their derivatives.

\par The key result is that we can compute a set in which lives the support of $New(f)$:
\begin{lemma}\label{recursive box} Let $f\in W_F^{(0)}$. Then, we can effectively find an integer $E(f)$ such that for all $\overline{y}\in \Zp^m$, either $f(\overline{X},\overline{y})$ is identically zero or $New(f(\overline{X},\overline{y}))\subseteq B_{\max} (E(f))$. 
\end{lemma}
In this lemma, $B_{\max} (E)$ denotes the set $\{I\in \R^n\mid\ \max_k\{|i_k|\}\leq E\}$. Note also that we have identified $New(f)$ and its support.
\begin{proof}
Let us recall that an element of $New(f)$ is the projection of a set $vert_\nu(f)$ (for $\nu\in \R^n$, $\nu=trop(\overline{x})$ for some $\overline{x}\in (\mathcal{O}_p^*)^n$) i.e.\ is the set of indexes $J$ such that $v(a_J(\overline{y}))+\langle\nu ,J\rangle$ reaches the minimum of the set $\{v(a_I(\overline{y}))+\langle\nu ,I\rangle;\ I\in \N^n\}$ for some $\nu\in [0,\infty)^n$. So, it is sufficient to show that for all $\nu\in [0,\infty )^n$ the projection of the set $vert_\nu(f)$ is contained in $B_{\max} (E(f))$ for suitable (computable) $E(f)$.
\par As $f\in W_F^{(0)}$, we know that there exists $d(f)$ (computable) such that for all $\lvert I\rvert\geq d(f)$,
 $$a_I(\overline{Y}) = \sum_{|J|<d(f)}b_{IJ}(\overline{Y}) a_J(\overline{Y}), $$
where $b_{IJ}\in W_F$ (the Weierstrass system generated by the $\mathcal{L}_F$-terms) with $\|b_{IJ}\|<1$.
Fix $\overline{y}\in \Zp$ and assume $f(\overline{X},\overline{y})\not\equiv 0$ i.e.\ $a_I(\overline{y})\not=0$ for some $|I|<d(f)$. First, let us remark that for all $I$ such that $i_1,\cdots, i_n\geq d(f)$, for all $\overline{x}\in (\mathcal{O}_p^*)^n$, we can find $J$ with $|J|< d(f)$ such that
\begin{align*}
v(a_I(\overline{y}))+\langle I,trop(\overline{x})\rangle &\geq  \min_{|K|< d(f)}\{ v(b_{IK}(\overline{y}))+v(a_K(\overline{y})) +\langle K,trop(\overline{x})\rangle\}\\
&> v(a_J(\overline{y}))+\langle J,trop(\overline{x})\rangle. 
\end{align*}
If $n=1$, take $E(f)=d(f)$ and we are done by the above inequality.
\par In the general case, we already know by the above inequality that no index $I$ that satisfies $i_1,\cdots, i_n\geq d(f)$ can be a point of $vert_\nu(f)$. It remains to bound indexes in $vert_\nu(f)$ with at least one coordinate less than $d(f)$.
\par  Fix $1\leq k\leq n$ and $1\leq s\leq d(f)$. Fix a coefficient $I$ whose $k$th coordinate is $s$. Then, $a_I(\overline{y})\overline{X}^I$ is the $(i_1,\cdots, i_{k-1},s,i_{k+1},\cdots, i_n)$th coefficient of the function $f_{s,k}(\overline{X},\overline{y})X_k^s$ where
 $$f_{s,k}(\overline{X},\overline{y})= (1/s!)\frac{\partial^s f}{\partial x_k^s}(X_1,\dots, X_{k-1},0,X_{k+1},\cdots, X_n,\overline{y}).$$
  Then, as $f_{s,k}\in W_F^{(0)}$ (by hypothesis this is closed under derivation), there is $d(f,s,k):=d(f_{s,k})$ such that for all $I$ with $\max_{j\not=k} \{i_j\}\geq d(f,s,k)$,
$$v(a_I(\overline{y}))+\sum_{l\not= k} i_l v(x_l) >\min \{v(a_{(j_1,\cdots, j_{k-1},s,j_{k+1},\cdots, j_n)}(\overline{y}))+\sum_{l\not= k} j_l v(x_l)\}.$$
where the min is taken in $\{J': |J'|=|(j_1,\cdots, j_{k-1},j_{k+1},\cdots, j_n)|<d(f,s,k)\}$. We set:
 $$E'(f)=\max_{k\leq n}\max_{ s\leq d(f)}\{d(f,k,s),d(f)\}.$$
 If $n=2$, we can take $E(f)=E'(f)$. Otherwise, we can compute $E(f_{s,k})$ for all $s\leq d(f)$ and $k\leq n$ by induction: we proceed like above with $f=f_{s,k}$. Then, we take $E(f)=\max_{s,k}\{E(f_{s,k}), E'(f)\}$.
\end{proof}
\begin{Remark}Note that in the above lemma, we can make vary the parameter $\overline{y}$ over $\mathcal{O}_p^m$. Then, it does not change the bound $E(f)$. This is also true for all the below result: the bounds we find also works if the parameters vary over $\mathcal{O}_p$ instead of $\Zp$.
\end{Remark}

We can now bound effectively the number of roots of a system $f$ with isolated tropicalization.
\begin{lemma}\label{effective isolated roots} Let $f=(f_1,\cdots, f_n)\subset W_F^{(0)}$. Then, one can compute integers $D_1$ and $D_2$ (depending only on $f$) such that for all $\overline{y}\in \Zp^m$, either $\bigcap V(f_i(\overline{X},\overline{y}))$ is infinite, or $\bigcap Trop(f_i(\overline{X},\overline{y}))\cap\R^n$ has less than $D_1$ isolated points and for each such a point $\nu$, the cardinality of $\bigcap V(f_i(\overline{X},\overline{y}),\{\nu\})$ is less than $D_2$. 
\end{lemma}
In particular, under these hypotheses, whenever the system $f$ has finitely may solutions in $(\mathcal{O}_p)^n$, it has at most $D_1\cdot D_2$ solutions in $(\mathcal{O}_p^*)^n$ with isolated tropicalization with positive valuation by Theorem \ref{Rabinoff 11.7}.
\begin{proof}
Assume that we have chosen $\overline{y}$ such that the number of solutions of the system is nonzero and finite. Then, by Lemma \ref{recursive box}, $New(f_{i,\overline{y}})$ is contained in $B_{\max} (E(f_i))$. So, for all $i$ and $\nu$, $\check{\gamma}_\nu(f_i):=\check{\gamma}_\nu(f_i(\overline{X},\overline{y}))\subset B_{\max} (E(f_i))$. As $MV$ is monotonic, 
$$MV(\check{\gamma}_\nu(f_1),\cdots ,\check{\gamma}_\nu(f_n))\leq MV(B_{\max} (E(f)), \cdots , B_{\max} (E(f)))= E(f)^n,$$
where $E(f)=\max_i E(f_i)$. Take $D_2=E(f)^n$. By Theorem \ref{Rabinoff 11.7}, $D_2$ satisfies the conditions of our lemma.
\par Let us recall that the points of $\bigcap Trop(f_i(\overline{X},\overline{y}))\cap \R_{>0}^n$ are determined by a system of linear equations. Each equation corresponds to an half-hyperplane contained in $Trop(f_i(\overline{X},\overline{y}))\cap\R_{>0}^n$ (determined by some $\gamma_\nu$). As these half-hyperplanes are in bijection with the faces of $New(f_i)$ (the $\check{\gamma}_v$'s, see Proposition \ref{Rabinoff prop 8.6.2}), we can bound the number of systems:
\par  Consider the  polygon contained in $B_{\max} (E(f_i))$  with the maximal number of faces (say this polygon has $d_i$ faces). Note that $d_i$ is computable. Then, $Trop(f_i(\overline{X},\overline{y}))\cap\R^n$ has at most $d_i$ half-hyperplanes. So, the number of isolated points contained in the intersection of all $Trop(f_j(\overline{X},\overline{y}))\cap \R^n$ is no more than $\prod_i d_i$. We define $D_1$ to be the product of all $d_i$'s.
\end{proof}

\par In general, we have points in the tropicalization of the system that are not isolated. But in fact, after a sufficiently small perturbation, the system can be reduced to this case. The next results and definitions come from \cite{Rabinoff}:
\begin{definition}
Let $P=\bigcap_i\{v\in \R^n\mid\ \langle u_i,v\rangle \leq a_i\}$ be a polyhedron in $\R^n$. A \emph{$\varepsilon$-thickening} of $P$ is a polyhedron of the form
$$ P' = \bigcap_i\{v\in N_\mathbf{R}\mid\ \langle u_i,v\rangle \leq a_i+\varepsilon\}.$$
More generally, if $\Pi$ is a polyhedral complex, a \emph{thickening} $\mathcal{P}$ of $\Pi$ is a collection of polyhedra of the form $\mathcal{P}=\{P'\mid\ P\in \Pi\}$, where $P'$ is a thickening of $P$. We set
$$ \lvert\mathcal{P}\rvert=\bigcup P'\qquad \mbox{and}\qquad int(\mathcal{P})=\bigcup int(P'),$$
where $int(P')$ denotes the interior of $P'$.
\end{definition}
So far, we use series from $\Zp\{\overline{X}\}$ in our language.  
We will now assume that the terms in our language are in $\Zp\langle P\rangle$ for some $P=\prod [r_i,\infty)$ which contains $P_0:=\prod [0,\infty)$ in its interior i.e.\ the series are overconvergent.
\par Let $f_1,\cdots, f_n\in \Zp\langle P\rangle$ be a system of overconvergent series. Let $C$ be a connected component of $\bigcap_i Trop(f_i,P)$ and $C_0$ be its restriction to $P_0$. Then, if we apply a small perturbation to the system, the component $C_0$ becomes a finite set of point:
\begin{lemma}\label{Rabinoff 12.5}
Let $C$ be a connected component of $\bigcap Trop(f_i,P)$. Then there is $\delta$, $P'$ a $\delta$-thickening of $P_0$ contained in $P$ and $\mathcal{P}$ a thickening of $C_0$ contained in $P'$ such that $\lvert\mathcal{P}\rvert\cap \bigcap_i Trop(f_i,P')=C'$ where $C'=C\cap P'$. There also exist $\overline{v}_1,\cdots , \overline{v}_n\in \N^n$ and $\varepsilon\in\Q_{\geq 0}$ such that for all $t\in (0,\varepsilon]$, the intersection 
$$ \lvert\mathcal{P}\rvert\cap \bigcap_i \Big(Trop(f_i,P')+t\overline{v}_i\Big)$$
is a finite set of points contained in $int(P')$. Furthermore, each of these point is determined by the intersection of affine polyhedra $\gamma_v$ contained in the tropicalizations $ Trop(f_i,P')$.
\end{lemma}
This follows from the definitions and from the proof of Lemma 12.5 in \cite{Rabinoff}. In fact, we know that $Trop(f_i,P)$ is the finite reunion of half-hyperplanes. After perturbation, the hyperplanes of each $Trop(f_i)$ intersect in at most one point. Indeed, up to a small perturbation, the intersection of $n$ hyperplanes in $\R^n$ is either empty or one point. These are the set of point in the above interesection. Furthermore, if $\delta$ is small enough, we can assume that $P'$ does not contains branching points on its boundary (i.e.\ the limit points of the intersection when $t$ tends to zero). In that case, the points in the above intersection are defined by intersection of half-hyperplanes in each $Trop(f_i,P')$ and are in the interior of $P'$.
\begin{Remark}
Note that $\bigcap Trop(f_i,P_0)$ has a finite number of component. We can apply the above lemma to each component so that $\lvert\mathcal{P}\rvert\cap \bigcap_i \Big(Trop(f_i,P')+tv_i\Big)$ is a finite set of points contained in the interior of $P'$ where $\mathcal{P}$ is a thickening of $P_0$.
\end{Remark}
\par We fix $t\in \Q$ and $\xi$ in some algebraic extension $K$ of $\Qp$ such that $v(\xi)=t$. Let $\overline{v}\in \N^n$.
We denote by $\widetilde{f}$ the image of the map:
$$
\begin{array}{rll}
	K\langle P\rangle &\longrightarrow &K \langle t\overline{v}+P\rangle\\
	f(x_1,\cdots , x_n)&\longmapsto & f(x_1\xi^{-v_1},\cdots , x_n\xi^{-v_n}).
\end{array}
$$
Then, $Trop(\widetilde{f})=Trop(f)+t\overline{v}$. Let us remark that $Trop(\widetilde{f})$ and $New(\widetilde{f})$ are independent of the choice of $\xi$ with $v(\xi)=t$ (as these sets are determined uniquely by the valuations of the coefficients of $\widetilde{f}$).
\par Let $f_1,\cdots, f_n$ be overconvergent series in $W_F^{(0)}$. Then, by Lemma \ref{Rabinoff 12.5}, the intersection of the tropicalization of the $\widetilde{f_i}$ is a finite set of points (for any suitable choice of $t,\overline{v}$). By Lemma \ref{effective isolated roots}, we can give an upper bound for the number of solutions of the system $(\widetilde{f_1},\cdots, \widetilde{f_n})$. This upper bound is effective if we assume that $W_F^{(0)}$ has an effective extended Weierstrass bound. For let us remark that if we replace $P=\prod [r_i,\infty)$ by a $\varepsilon'$-thickening of $P_0$ for some $\varepsilon'$ small enough, the extended Weierstrass bound is an upper bound for $d(\widetilde{f_i})$. This will be sufficient to estimate the number of solution of the system $(f_1,\cdots,f_n)$ in $\mathcal{O}_p$.

\begin{definition} Let $f_1,\cdots, f_n\in W_F^{(0)}$ overconvergent. If $\nu\in \cap Trop(f_i)$ is isolated, we define
$$i(\nu,Trop(f_1),\cdots, Trop(f_n)):=MV(\check{\gamma_1},\cdots, \check{\gamma_n})$$
where $\check{\gamma_i}:=\pi (vert_\nu(f_i))\in New(f_i)$.
\par Let $\mathcal{P}$ as in the remark after Lemma \ref{Rabinoff 12.5}. We define
$$i(\mathcal{P} ,Trop(f_1),\cdots , Trop(f_n))=\sum_\nu i(\nu, Trop(\widetilde{f}_1),\cdots, Trop(\widetilde{f}_n))$$
where the sum is taken over all $\nu\in \lvert\mathcal{P}\rvert\cap \bigcap_i \Big(Trop(f_i,P)+t\overline{v}_i\Big)$.
\end{definition}

\begin{lemma}\label{effective roots in a connected component}
Let $f_1,\cdots, f_n\in W_F^{(0)}$ where $W_F^{(0)}$ has an extended effective Weierstrass bound. Then, we can compute $T$ such that for all $\overline{y}$,
$$i(\mathcal{P} ,Trop(f_1),\cdots , Trop(f_n))\leq T$$
where $T$ does not depends on any of the choice $\mathcal{P}$, $t,v_i$.
\end{lemma}
\begin{proof}
Let $\widetilde{f}_i(\overline{X})=f_i(\overline{X}t^{-v_i})$ where $t\in \mathcal{O}_p$ has positive sufficiently small valuation. After perturbation, there is only finitely many isolated points by Lemma \ref{Rabinoff 12.5}. Then as in \ref{effective isolated roots}, one can compute an upper bound for the number of roots with isolated tropicalization in $P'$ and for the number of isolated points. For we replace $d(f_i)$ by $d(\widetilde{f}_i)$ in the proof of Lemma \ref{recursive box}. This latter is computable as $W_F^{(0)}$ has an extended effective Weierstrass bound.  We obtain a computable upper bound for $i(\mathcal{P} ,Trop(f_1),\cdots , Trop(f_n))$.
\end{proof}

We relate now the solution of the system $(\widetilde{f}_1,\cdots, \widetilde{f}_n)$ to the solution of $(f_1,\cdots, f_n)$.

\begin{theorem}\label{Rabinoff 12.11} Let $f_1,\cdots , f_n\in W_F^{(0)}$ overconvergent. Assume that the system $V(f_1,\cdots, f_n)$ has a finite number of solutions with tropicalization in $P=\prod (-r_i,\infty]$ (for all $r_i<0$ small enough) and no solutions with zero coordinates.
$$\left| \bigcap_i V(f_i;P')\right|\leq i(\mathcal{P}, Trop(f_1)\cdots Trop(f_n))$$
for all $\mathcal{P}, P'$ with $\varepsilon,\delta$ given in Lemma \ref{Rabinoff 12.5} small enough.
\end{theorem}
\begin{proof}
First, taking $r_i$ small enough, we may assume that $f_j\in\Zp\langle P\rangle$ with $P=\prod_i [-r_i,\infty)$, that the tropicalization of any zero of the system $f_1,\cdots, f_n$ is in $\prod (r_i,T_i)$ (for some $T_i$ large enough) and $Trop(f_i)$ has no branching point on the boundary of $P$. Let $g_i(\overline{X},T)=f_i(X_1T^{v_{1i}},\cdots ,X_nT^{v_{ni}})$ where $\overline{v}_i$ is given by Lemma \ref{Rabinoff 12.5}. Let $Y=\cap_i V(g_i)$ and $Y_t= \cap_i  V(g_{i,t})$ with $g_{i,t}:=f_i(X_1t^{-v_{1i}},\cdots, X_nt^{-v_{ni}})$. Note that $g_{i,1}=f$. So, there are $r'_i=-\delta$, such that $Trop(Y_1)\subset P':= \prod (r'_i,T_i)$. On the other hand, for all $t$ with $v(t)<\varepsilon$ ($\varepsilon$ given by Lemma \ref{Rabinoff 12.5}), $Trop(Y_t)$ is contained in $|\mathcal{P}|\cap \bigcap_i\left( Trop(f_i,P')+tv_i\right)$  i.e.\ in the interior of $P'$ (for $\varepsilon$ small enough, $T_i$ large enough).
So, by \cite{Rabinoff} Theorem 9.8, the cardinality of $Y_t$ (with tropicalization in $P'$) does not depends on $t$. Furthermore, as each point in $Trop(Y_t)$ is isolated, by Theorem \ref{Rabinoff 11.7}, the number of points in $Y_t$ (with $v(t)>0$) with valuation $\nu$ is bounded by the mixed volume of the Newton polygons corresponding to $\nu$.
 Then $\left| \bigcap_i V(f_i;P')\right|\leq i(\mathcal{P}, Trop(f_1)\cdots Trop(f_n))$.
\end{proof}

With this theorem, we are now able to prove the main theorem of this section:

\begin{theorem}\label{effective bound of system} Let $F$ be a family of restricted analytic function overconvergent so that $W_F^{(0)}$ has an extended effective Weierstrass bound.
Let $f=(f_1,\cdots, f_n)\in W_F^{(0)}$. Then, there exists $S(f)$ computable such that for all $\overline{y}\in \Zp^n$, either the system $(f_1(\overline{X},\overline{y}),\cdots, f_n(\overline{X},\overline{y}))$ has infinitely many roots or it has less than $S(f)$ roots with tropicalization in $P=\prod(-r_i,\infty)$ for all $r_i>0$ with $|r_i|$ small enough.
\end{theorem}
\begin{proof} First, let us remark that if $\cap V(f_i,P)$ has a finite number of solutions, then up to a change of variable of the type $X_i\rightarrow X_i-s_i$, we can assume that it has no solution with zero coordinate. We add extra-parameters $\overline{s}$ and replace $f_i$ by $f'_i(\overline{X},\overline{S}):=f_i(\overline{X-S})$.
By the above theorem, if $\cap V(f'_i,P)$ is finite then it is bounded by $i(\mathcal{P}, Trop(f'_1)\cdots Trop(f'_n))$. The result follows now from Lemma \ref{effective roots in a connected component}
\end{proof}

\begin{Remark}
Let $f_1,\cdots, f_{n+m}\in \Zp\{X_1,\cdots, X_n,\overline{Y}\}[X_{n+1},\cdots, X_{n+m}]$ convergent on $B\times\Cp$ and satisfying the hypotheses of the above theorem. Then we can compute a bound for the number of solutions of the system in $(\mathcal{O}_p)^n\times (\C_p)^m$. Indeed, in this case, the size of the box computed in Lemma \ref{recursive box} with respect to the variable $X_{n+i}$ is determined by the degree of $f_k$ as polynomial in $X_{n+i}$. Therefore, using Theorem \ref{Rabinoff 12.11}, for all $r_i$, we can compute a bound for the number of solution with tropicalization in $P \times\prod_i[r_i,\infty)$. Furthermore, we remark that the bound $S(f)$ obtained in this case is independent on the choice of $r_i$ (as so is the box from Lemma \ref{recursive box}; in the polynomial case it depends only on the degree of the polynomials) which means that it is a bound for the number of solutions in $(\mathcal{O}_p)^n\times (\C_p)^m$.
\end{Remark}

\section{Effective model-completeness}\label{Effective model-complenetess}

\par We can now prove the second main theorem:

\begin{theorem}\label{effective model completeness of L_F} Let F be an effective family of restricted analytic functions such that the set of $\mathcal{L}_F$-terms is closed under derivation. Let $\widetilde{F}$ be the extension of $F$ by all decomposition functions of elements in $F$. Assume that each $\mathcal{L}_{\widetilde{F}}$-term is overconvergent and that $W_F^{(0)}$ has an extended effective Weierstrass bound.
\par Then, the theory of $\mathbb{Z}_{p,\widetilde{F}}$ is effectively strongly model-complete in the language $\mathcal{L}_{\widetilde{F}}$. 
\end{theorem}
\begin{proof}
For, as we have seen in section \ref{Effective Weierstrass system}, it is actually sufficient to prove that $W_{\widetilde{F}}$ is an effective Weierstrass system. Let $f\in W_{\widetilde{F},n}^{(k)}$. We have to show that $S(f)$ is computable. We proceed by induction on $k$ and we show that for any $f\in W_{\widetilde{F},n}^{(k)}$, $S(g)$ is computable where $g=f$ or one of its derivatives. The basic step of the induction follows immediately from our hypothesis.
\par So assume that for all $n$, for all $k\leq m$ and for all $g\in W_{\widetilde{F},n}^{(k)}$, $S(g)$ and all its derivatives can be bounded. Let $H\in W_{\widetilde{F},n}^{(m+1)}$. We want to compute $S(H)$ (or more generally, $S(G)$ where $G$ denotes a derivatives of $H$). By definition of the Weierstrass system generated by the $\mathcal{L}_{\widetilde{F}}$-terms, $H$ is a polynomial combination one of the following possibilities:
\begin{enumerate}[(a)]
	\item $h\in W_{\widetilde{F},n}^{(m)}$. In that case, we can compute $S(h)$ by inductive hypothesis.
	\item There are $f\in W_{\widetilde{F},n}^{(m)}$ and a permutation $\sigma$ such that $h(\overline{X})=f(X_{\sigma(1)},\cdots , X_{\sigma(n)})$. In that case, we can compute $S(f)$ by inductive hypothesis and $S(h)=S(f)$. The same holds for any derivative of $h$.
	\item There is $f\in W_{\widetilde{F},n}^{(m)}$ such that $f$ is invertible in $\Zp\{\overline{X}\}$ and $h=f^{-1}$. In that case, $S(f)=S(h)=1$. Also, $S\left(\frac{\partial h}{\partial X_i}\right)=S\left(-\frac{\partial f}{\partial X_i}h^2\right)=S\left(\frac{\partial f}{\partial X_i}\right)$ (this is also bounded by $S(f^2\frac{\partial f}{\partial X_i})$ and similarly for the higher derivatives.
	\item There are $f,g\in W_{\widetilde{F},n}^{(m)}$ such that $h=f/g(0)$. In that case, we can compute $S(f)$ by inductive hypothesis and $S(h)=S(f)$. The same holds for any derivative of $h$.
	\item There are $f\in W_{\widetilde{F},n+1}^{(m)}$ of order $d$ in $X_{n+1}$ and $g\in W_{\widetilde{F},n+1}^{(m)}$ such that $h$ is one of the functions $a_0, \cdots , a_{d-1}\in \Zp\{X_1,\cdots , X_{n}\}$ or $Q\in \Zp\{X_1,\cdots , X_{n+1}\}$ given by the Weierstrass division theorem.
\end{enumerate}
Note that in case (a)-(d), one also get that $d(h)$ is determined by $f$.
In the last case, $h$ (or any of its derivatives) is actually determined by a system of equations (see facts  \ref{strong def W-coeff}  to \ref{strong def W-deriv1} in Proposition \ref{strong def W-syst}). More generally, let $h(\overline{X}) = P(\overline{X},a_0(\overline{X}),\cdots, a_s(\overline{X}))$ where $P$ is any polynomial with coefficients in $\Z$. Then,
\begin{claim}\label{claim thm effective model completeness of L_F}
$S(h)$ and $S(h')$ can be bounded effectively where $h'$ is a derivative of $h$.
\end{claim}
\begin{proof}
We want to compute a bound of $S(h)$. Let $h(Z,\overline{Y}) = P(Z,a_0(Z,\overline{Y}),\cdots, a_s(Z,\overline{Y}), \overline{Y})$. We want to bound the number of solutions of the equation $h(Z,\overline{y})=0$ for any $\overline{y}\subset\Zp^{n+k-1}$ such that this number is finite (where $Z=X_1$ and $\overline{y}$ denotes now $(x_2,\cdots, x_{n-1},y_1,\cdots, y_k)$).  Fix $\overline{y}$ such that the number of roots is finite. Note that if we add an extra-parameters, we can assume that all roots are nonsingular. Let us remark that $z$ is a solution of $h(Z,\overline{y})=0$ if $z,t_0,\cdots , t_s, a_0,\cdots a_s$ are solutions of the system of equations:
$$\left\{
\begin{array}{l}
f(t_0,z,\overline{y})=0\\
\vdots \\
f(t_s,z,\overline{y})=0\\
\left(\begin{array}{cccc}
1 & t_0 &\cdots & t_0^{s}\\
\vdots & \vdots && \vdots \\ 
1 & t_s &\cdots & t_s^{s}
\end{array}\right)
\left(\begin{array}{c}
a_0\\
\vdots\\
a_{s}
\end{array}\right) =
\left(\begin{array}{c}
g(t_0,z,\overline{y})\\
\vdots\\
g(t_s,z,\overline{y})
\end{array}\right)\\
P(z,a_0,\cdots, a_s, \overline{y})=0
\end{array}\right.$$ 
if $t_i\not=t_j$ for all $i\not=j$. To make sure that this last condition is satisfied, we introduce the variables $t_{ij}$ $0\leq i<j\leq s$ and add to the system the equations:
$$ t_{ij}\cdot(t_i-t_j)-1=0.$$
\par Note that this system has finitely many solutions in $(\mathcal{O}_p)^{2s+3}\times (\Cp)^{(s^2+s)/2}$ if $h(Z,\overline{y})$ has finitely many solutions in $\mathcal{O}_p$. Conversely, the number of solution of $h(Z,\overline{y})$ is equal to the sum of the number of the solutions of the different systems taking in account all possible multiplicities of the  $t_i$'s.
\par So, the number of solutions of $h(Z,\overline{y})$ is determined by the sum of the number of solutions of systems $(f_1^{( i)},\cdots f_{N_i}^{( i)})$ where $f_j^{(i)}\in W_{\widetilde{F}}^{(m)}$ (and $i$ varies over all possible multiplicities). Going down by induction (by Proposition \ref{strong def W-syst}, the zeros of any element of $W^{m+1}$ is determined by a system of functions in $W^{(m)}$), we can actually assume that the functions $f_j^{(i)}$ are in $W_{\widetilde{F}}^{(0)}$ (i.e.\ are $\mathcal{L}_{\widetilde{F}}$-terms). So, by Theorem \ref{effective bound of system}, one can compute a bound $S_i$ for the number of solutions of the system $(f_1^{( i)},\cdots f_{N_i}^{( i)})$. Take $S=\sum S_i$. Then $S$ is a bound for $d(h)$.
\par Let $h'$ be a derivative of $h$. We can compute $d(h')$ in a similar way using the definitions given in the facts \ref{strong def W-unit}  and \ref{strong def W-deriv1} in Proposition \ref{strong def W-syst}.
\end{proof}
The cases where $h$ is equal to a function $Q$  like in (e) or one of its derivative is obtained similarly using systems given in Proposition \ref{strong def W-syst}. With the same argument, we can compute $S(H)$ for a general function in $W_{\widetilde{F},n}^{(m+1)}$. Indeed, $H$ is just a polynomial combination of functions of type (a)-(e) and so is also determined by a system of equations whose functions (and their derivatives) are $\mathcal{L}_F$-terms.

\end{proof}

\section{Application: effective model-completeness of the $p$-adic exponential ring}\label{exponential ring}

\par Let us recall that the natural exponential function $exp(x)=\sum x^n/n!$ is convergent iff $v(x)>1/(p-1)$. Unlike the real field, the $p$-adic field does not carry a natural structure of exponential field. Yet we can use $exp(X)$ to define a structure of exponential ring: Let $E_p$ be the map $\Zp\longrightarrow \Zp: x\longmapsto exp(px)$ (if $p\not=2$, in the other case, we set $E_2(x)=exp(4x)$). It induces a structure of exponential ring on $\Zp$ i.e.\ $(\Zp,+, -, \cdot, 0,1,E_p)$ is a ring and $E_p$ is a morphism of groups from $(\Zp,+)$ to $(\Zp^\times, \cdot)$. This structure is a natural equivalent to the structure $(\R, +, -, \cdot, 0, 1, <, exp\mathord{\upharpoonright}_{[-1,1]})$. It is known that the real exponential field is decidable if Schanuel's conjecture is true \cite{Macintyre-Wilkie}. We use the results of this paper as a first step to a $p$-adic equivalent result. In this section, we apply our results to the set $F=\{E_p\}$. In this case we denote the language $\mathcal{L}_{F}$ by $\mathcal{L}_{exp}$.
\begin{Remark} For the rest of this section, we will assume $p\not=2$. The case $p=2$ should be obvious: we have to replace $p$ by $4$ when relevant.
\end{Remark}
\par The model-completeness in this case was first done by A. Macintyre in \cite{Macintyre4}. A first easy observation is that we don't need to add all decomposition functions. Indeed, let $K=\Qp (\alpha)$ and $V=\Zp[\alpha]$. As, $E_p(\sum \alpha^ix_i)=\prod E_p(\alpha^ix_i)$, it is sufficient to add to our language the functions $c_{i,j}$ such that:
$$E(\alpha^i x)=c_{0,i}(x)+\cdots + c_{d-1,i}(x)\alpha^{d-1}.$$
Following Macintyre's terminology, we call these functions \emph{trigonometric functions}.
Let $\mathcal{L}_{pEC}$ be the expansion of the language $\mathcal{L}_{exp}$ by all trigonometric functions of $K_n$ (where  $(K_n)_{n\in\N}$ is the tower of extensions defined in section \ref{Decomposition functions}). Let $\Z_{pEC}$ be the structure with underlying set $\Zp$ and natural interpretations for the symbols of $\mathcal{L}_{pEC}$. Then, by Theorem \ref{strong model-completeness},
\begin{theorem}[Macintyre \cite{Macintyre4}] $Th(\Z_{pEC})$ is strongly model-complete.
\end{theorem}
Using Theorem \ref{effective model completeness of L_F}, we will now prove that this model-completeness is effective.
First, let $W^{(0)}_E$ be the set of polynomial combinations (over $\Z$) of variables $X_k$, exponential $E_p(X_k)$ and decomposition functions $c_{i,j}(X_k)$ for all $i,j,k$. Let $W_E$ be the Weierstrass system as defined in Section \ref{Weierstrass system generated by} where we replace the set of $\mathcal{L}_F$-terms by $W^{(0)}_E$ at step zero of the definition. Let us remark that as Weierstrass systems are closed under composition $W_E$ is equal to the Weierstrass system generated by the $\mathcal{L}_{pEC}$-terms. On the other hand the condition that $W_F^{(0)}$ has an extended effective Weierstrass bound in Theorem \ref{effective model completeness of L_F} is now easier to prove as it is sufficient to prove it for elements of $W^{(0)}_E$ rather that for the set of $\mathcal{L}_{\widetilde{F}}$-terms. We give now a proof of this condition. Let us start with a simpler computation:
\begin{lemma}\label{Exponential Weierstrass bound}
There is a recursive function $d:\Z[\overline{X},\overline{Y},E_p(\overline{X}),E_p(\overline{Y})]\rightarrow \N$ such that for all $f(\overline{X},\overline{Y})=\sum_L c_l(\overline{X})\overline{Y}^L\in \Z[\overline{X},\overline{Y},E_p(\overline{X}),E_p(\overline{Y})]$, for all $L$ with $|L|\geq d(f)$ and $M$ with $|M|<d(f)$, there is $b_{LM}(\overline{X})\in \Zp\{\overline{X}\}$ with $\|b_{LM}\|<1$ and $\|b_{LM}\|\rightarrow 0$ as $|L|\rightarrow \infty$ such that 
$$c_L(\overline{X}) = \sum_{|M|<d(f)}b_{LM}(\overline{X})c_M(\overline{X}).$$
\end{lemma}

\begin{proof} Let $|\overline{X}|=n$ and $|\overline{Y}|=m$.
Let $f(\overline{X},\overline{Y})=\sum_{|I|\leq D}a_I(\overline{X},\overline{Y})E_p(\langle(\overline{X},\overline{Y});I\rangle)$ (where $\langle\cdot;\cdot\rangle$ denotes the scalar product and $D$ is the polynomial degree of $f$) with $a_I(\overline{X},\overline{Y})=\sum_{|J|\leq D} a_{IJ}(\overline{X})\overline{Y}^J\in \Z[\overline{X},\overline{Y}]$. Then,
$$c_L(\overline{X})=\frac{1}{L!}\frac{\partial^L f}{\partial \overline{Y}^L}(\overline{X},\overline{0}).$$
We will use classical multi-index notation: if $K,L\in \N^n$,
$$K!:=k_1!\cdots k_n!\qquad \binom{L}{K}:=\binom{l_1}{k_1}\cdots \binom{l_n}{k_n} \qquad L^K:=l_1^{k_1}\cdots l_n^{k_n},$$
$$L-K:=(l_1-k_1,\cdots l_n-k_n).$$
By Leibniz rule, for all $L$ with $|L|>D$
\begin{align*}
c_L(\overline{X}) &=\frac{1}{L!} \sum_{|I|\leq d}\sum_{K\leq L} \binom{L}{K} \frac{\partial^K a_I}{\partial \overline{Y}^K}(\overline{X},\overline{0}) \frac{\partial^{L-K} E_p(\langle(\overline{X},\overline{Y});I\rangle)}{\partial \overline{Y}^{L-K}}(\overline{X},\overline{0})\\
 &= \sum_{|I|,|K|\leq D}a_{IK}(\overline{X}){I'}^{L-K}\frac{p^{|L|-|K|}}{(L-K)!}E_p(\langle\overline{X},I'\rangle);
\end{align*}
where in the second line, we sum over $|K|\leq D$ as $a_{IK}=0$ for all $|K|>D$ ($I'$ denote the $m$ last coordinates of $I$).
Let $d(f)\in \N$ (to be defined later) and $V_{LT}\in \Zp$ (also to be defined later) with $\|T\|=\min\{T_i\}>D$, then
$$c_L(\overline{X})-\sum_{|T|<d(f)}V_{LT}c_T(\overline{X})=$$
$$\sum_{|I|,|K|\leq D} a_{IK}(\overline{X})E_p(\langle\overline{X},I'\rangle)\left[\frac{p^{|L|-|K|}}{(L-K)!}{I'}^{L-K}-\sum_{|T|<d(f)}V_{LT}\frac{p^{|T|-|K|}}{(T-K)!}{I'}^{T-K} \right].$$
If we find $V_{LT}$ such that the part between the brackets $[\cdots]$ vanishes, $v_p(V_{LT})>0$ and tends to infinity as $L$ tends to infinity we are done (provided that $d(f)$ is computable). First, we remark that the first condition is true provided that $V_{LT}$'s are solutions of the linear system:
$$\sum_{|T|<d(f)}V_{LT}\frac{p^{|T|-|K|}}{(T-K)!}{I'}^{T-K} = \frac{p^{|L|-|K|}}{(L-K)!}{I'}^{L-K}$$
for all $|I|,|K|\leq D$. We will pick $T_1,\cdots, T_{v}$ ($v$ is the number of equations) such that $|T_i|\leq d(f)$ for all $i$. We will set $V_{LT}=0$ for all $T\not=T_i$. The last variables determines a linear system of $v$ linear equations with $v$ unknowns. Let $A$ be the matrix corresponding to this system i.e. the elements of the matrix are $\frac{p^{|T_i|-|K|}}{(T_i-K)!}{I'}^{T_i-K}$ , the lines are indexed over all $|K|,|I|<D$ and the columns over all $T_i$ ($1\leq i\leq v$).
\begin{claim} There are $T_i$, $\|T_i\|>C_i(f)$ $(1\leq i\leq v$)for some computable $C_i(f)$, such that $det\ A\not=0$.
\end{claim}
\begin{proof}
We prove it by induction on the number of variables taking advantage of the symmetries of $A$. If there is only one variable, this is obvious: $\frac{p^{|T|-|K|}}{(T-K)!}{I'}^{T-K}\not=0$. Notice that in the case where $I'$ has a zero coordinate (say its $k$th coordinate), then either the corresponding coordinate of $L$ is zero (and therefore of similarly for $T$), or the line in the linear system corresponding to $I'$ is trivial. In the first case, as the $k$th coordinate of $T$ is determined, we do an induction on length of $T$. In the second case, we can remove the line $I'$ from the system and a columns corresponding to $T_i$ (for any $i$ we pick). We set $V_{LT_i}$ to be $p$. So, in both case, we may assume that $I'$ as no zero coordinate.
\par We develop the determinant of $A$ along the last line: we get
$$det\ A=\sum_{|I|,|K|\leq D} (-1)^{|I|+|K|}\frac{p^{|T_v|-|K|}}{(T_v-K)!}{I'}^{T_v-K} det\ A_{KI},$$
where $A_{KI}$ denotes the corresponding minor. Note that by the symmetries of the matrix $A$, we will be able to apply the hypothesis of induction to $det\ A_{KI}$. So, we are done if we prove the following: assume that we have found $T_1,\cdots ,T_n$ (with $\|T_i\|>C_i(f)$ for some constant $C_i(f)$ for all $i\leq n$) such that any minor of the matrix $A$ that involves only $T_1,\cdots, T_n$ is non zero. $T_1,\cdots, T_n$ are now fixed. Let $A'$ be a minor involving only $T_1,\cdots, T_{n+1}$. We expand $det\ A'$ along the line of the variable $T_{n+1}$ then:
$$det\ A'=\sum_{I,K} R_{IK}\frac{p^{|T_{n+1}|-|K|}}{(T_{n+1}-K)!}{I'}^{T_{n+1}-K},$$
where $R_{SK}$ is (up to sign) the determinant of some minor involving $T_1,\cdots, T_n$; in particular, $R_{SK}\not=0$. We include everything that does not depends on $T_{n+1}$ in some new (nonzero) constant so that $det\ A'\not=0$ iff
$$\sum_{I,K} {R'}_{IK}\frac{{I'}^{T_{n+1}}p^{|T_{n+1}|}}{(T_{n+1}-K)!},\qquad (*)$$
is non zero (again ${R'}_{IK}\not=0$). Let $I'$ maximal in the lexicographic order among the indexes with at least one the coordinate maximal. Then if
$$p^{|T_{n+1}|}\sum_{K} {R'}_{IK}\frac{1}{(T_{n+1}-K)!},\qquad (**)$$
is non zero, in $(*)$, $\frac{{I'}^{T_{n+1}}}{(T_{n+1}-K)!}$ is the dominant term (for the real topology) as $T_{n+1}$ tends to infinity. Therefore, for $\|T_{n+1}\|>C$, it is non zero. The constant $C$ here depends on $|R'_{IK}|$ i.e.\ on the choice of $T_1,\cdots ,T_{n}$ but this can be done effectively: given $T_1,\cdots, T_n$, one can find effectively $C$ (and therefore $T_{n+1}$) such that $(*)$ is non zero. 
Note that we have to assume that the sum over $|I'|$ is not reduced to one term $I'=(1,\cdots, 1)$; in this latter case it is sufficient to prove that $(**)$ is nonzero.
\par For $\|T_{n+1}\|>C'$, $(**)$ is non zero, as its dominant term (for the real topology) corresponds to $K$ of minimal lexicographic order. One can find such constant $C'$ explicitety given fixed $T_1,\cdots, T_n$. Take $C_{n+1}(f)$ to be the max of $C,C'$ where we take all possible such constant for any choice of minor involving only $T_1,\cdots , T_n$. Take $T_{n+1}$ such that none of the minor $det\ A'$ vanishes (this can be done effectively: pick the smallest $T$ for the lexicographic order such that $\|T\|>C_{n+1}(f)$ and $T\not= T_i$ for $i=1,\cdots ,n$).
 This completes the inductive step and therefore the proof of the claim. 
\end{proof} Let $K_1(f)=\sum C_i(f)$ and $K_2(f)=K_1(f)+v+1$.
Now, we can assume $\|T_i\|>C_i(f)$ and $|T_i|< K_2(f)$ and that $det\ A\not=0$. For all $i$, we find that $V_{LT_i}= (det\ A)^{-1} A_i \frac{p^{|L|-|K|}}{(L-K)!}I^{L-K}$ ($A_i$ denote the $i$th line of $A^{-1}$).
\begin{claim}\label{estimate} $v_p(V_{LT_i})$ tends to infinity as $|L|$ tends to $\infty$. Futhermore for all $|L|>K_3(f)$ (for some computabe $K_3(f)$),  $v_p(V_{LT_i})>0$
\end{claim}
\begin{proof}
Let us give a rough estimate of the valuation of each of the following terms:
\begin{enumerate}[(a)]
	\item $v((det\ A)^{-1} A_i)$: as $|T_i|<K_2(f)$, $|I|,|K|\leq d$ it can take finitely many rational values (each of them can be computed). Let $G$ be the minimum of these values (we take the minimul over all values such that $det\ A$ is nonzero). Then, $v((det\ A)^{-1} A_i)\geq G$.
	\item $v(I^{L-K})$: as $|I|\leq D$, $v(i_j)\leq d$ for all $j$. So, the valuation is at least $-D|K|\geq -D^2$.
	\item $v(\frac{p^{|L|-|K|}}{(L-K)!})$: this is greater that $v(p^{|L|}/L!)-|K|$. By the classical evaluation of $v(L!)$ this is greater than $|L|(p-2)/(p-1)-|K|\geq |L|(p-2)/(p-1)-D$ (we assumed that $p\not=2)$.
\end{enumerate}
Putting everything together, we obtain that $v(V_{LT})\geq |L|(p-2)/(p-1)-D- D^2+G$. Surely, this tends to $\infty$ as $|L|\rightarrow \infty$. Furtermore, if $|L|>(-G+D+D^2)(p-1)/(p-2)$, this is positive. 
\end{proof}
Take $d(f)=\max\{(-G+D+D^2)(p-1)/(p-2),K_2(f)\}$ and $b_{LM}=V_{LM}$ as defined above ($V_{LM}=0$ if it is not determined by the linear system). Then all claims of the lemma are satisfied.
\end{proof}
Let us remark that we can generalise the above lemma in the following ways: Let $\Z[\overline{X}]^{EC}$ be the ring generated by $X_i, E_p(X_i)$ and $c_{jk}(X_i)$ for all $i\leq |\overline{X}|,j,k$. Then,
\begin{lemma}\label{d(f) pEC-term}
There is a recursive function $d:\Z[\overline{X},\overline{Y}]^{E,C}\rightarrow \N$ such that for all $f(\overline{X},\overline{Y})=\sum_L c_l(\overline{X})\overline{Y}^L\in \Z[\overline{X},\overline{Y}]^{E,C}$, for all $L$ with $|L|\geq d(f)$ and $M$ with $|M|<d(f)$, there is $b_{LM}(\overline{X})\in W_E$ with $\|b_{LM}\|<1$ and $\|b_{LM}\|\rightarrow 0$ as $|L|\rightarrow \infty$ such that 
$$c_L(\overline{X}) = \sum_{|M|<d(f)}b_{LM}(\overline{X})c_M(\overline{X}).$$
\end{lemma}
\begin{proof}
Let $f(\overline{X},\overline{Y})\in \Z[\overline{X},\overline{Y}]^{E,C}$. Any such elements is a polynomial combination of $\overline{X},\overline{Y},E_p(\alpha^j\overline{X}),E_p(\alpha^j\overline{Y})$. 
Indeed, let $c_j(X)$ be a trigonometric function function. Then, $c_j(X)=\sum_J \overline{u}^JE_p(\langle (X,\alpha);J\rangle$ (a finite sum, $\overline{u}$ a set of parameters). We define the series $f^+(\overline{X},\overline{Y},\overline{A},\overline{U})$ that is in the exponential polynomial such that $f^+(\overline{X},\overline{Y},\overline{\alpha},\overline{u})=f(\overline{X},\overline{Y})$. Assume that 
$$f(\overline{X},\overline{Y})=\sum_I a_I(\overline{X})\overline{Y}^I, $$
$$f^+(\overline{X},\overline{Y},\overline{A},\overline{U})=\sum_I a_I^+(\overline{X},\overline{A},\overline{U})\overline{Y}^I.$$
Then, $a_I(\overline{X})= a_I^+(\overline{X},\overline{Y},\overline{\alpha},\overline{u})$. By Lemma \ref{Exponential Weierstrass bound}, there is $d(f^+)$ and $b^+_{IJ}\in p\Zp\{\overline{X},\overline{A},\overline{U}\}$ such that 
$$a_I^+(\overline{X},\overline{Y},\overline{A},\overline{U})=\sum_{|J|<d(f^+)} a_J^+(\overline{X},\overline{Y},\overline{A},\overline{U})b^+_{IJ}(\overline{X},\overline{Y},\overline{A},\overline{U}).$$
Then, as $a_I(\overline{X})= a_I^+(\overline{X},\overline{Y},\overline{\alpha},\overline{u})\in \Zp\{\overline{X}\}$, we obtain $b_{IJ}\in p\Zp\{\overline{X}\}$ so that
$$a_I(\overline{X})=\sum_{|J|<d(f^+)} a_J(\overline{X})b_{IJ}(\overline{X}).$$
\end{proof}
The other generalisation is the computation of an extended Weirstrass bound:
\begin{prop}
The ring $W_E^{(0)}$ admits an extended effective Weierstrass bound.
\end{prop}
\begin{proof}
\par Let $f(\overline{X})\in \Z[\overline{X},E_p(\overline{X})]$. Let $t\in \mathcal{O}_p$ with positive valuation at most $1/(p-1)$; say $1/p^2$. 
Then, $f(t^{-1}\overline{X})=P(t^{-1}\overline{X},E_p(t^{-1}\overline{X}))$. Notice that $E_p(t^{-1}X_i)=exp(pt^{-1}X_i)$. Set $E_t(X):=exp(pt^{-1}X)$. Then, to prove the existence of the extended effective Weierstrass bound, we have to prove a version of Lemma \ref{Exponential Weierstrass bound} where we replace $E_p$ by $E_t$. The proof is similar. In claim \ref{estimate}, we have to estimate the valuation of $v((pt^{-1})^{|L|}/L!)$. This is possible as $v(t)=1/p^2<1/(p-1)$, so this valuation is guaranteed to tends to infinity as $|L|$ tends to infinity.
\par The same argument holds for elements of $W^{(0)}_E$ as in the last lemma.
\end{proof}

\begin{Remark} In \cite{Macintyre4}, A. Macintyre gives an algorithm that compute $S(f)$ (an upper bound for the number of roots in $\mathcal{O}_p$ for any choice of parameters such that it is finite) for any $\mathcal{L}_{pEC}$-term $f$. He uses an induction on the 'exponential height' for exponential terms and proceeds as in Lemma \ref{d(f) pEC-term} in general. This result also follows from the above proposition together with Theorem \ref{effective bound of system}. On the other hand, it seems that the bound obtained by Macintyre is sharper than ours.
\end{Remark}

\par Finally, by the last proposition and by Theorem \ref{effective model completeness of L_F},
\begin{theorem} $Th(\Z_{pEC})$ is effectively strongly model-complete.
\end{theorem}
So, the decidability of the full theory of $Th(\Z_{pEC})$ or of the $p$-adic exponential ring is reduced to the decision problem for $\mathcal{L}_{PEC}$-existential formula i.e.\ is there an algorithm which determines the truth value of existential formula in $\Zp$.
In a subsequent paper, the author will solve this problem assuming a $p$-adic version of Schanuel's conjecture.

\begin{acknowledgement} The results of this paper were part of the author thesis. The author would like to thanks his thesis supervisor A. Wilkie. He also thanks F. Point and A. Macintyre for their support during the writing process. The author is very grateful to the referee for pointing out a mistake in an earlier version of the draft and many helpful comments an suggestions.
\end{acknowledgement}

\bibliographystyle{plain}
\bibliography{Biblio_p-adic}
\noindent Nathana\"el Mariaule\\
Universit\'e de Mons, Belgium\\
\emph{E-mail address: Nathanael.MARIAULE@umons.ac.be}

\end{document}